%% file: super_convergence-arxiv.tex
\documentclass[10pt,reqno,final]{amsart}
\usepackage{amsfonts,latexsym,url,amsmath,amssymb,comment,graphicx,enumitem,pgfplots,pgfplotstable}

\excludecomment{commentforus}

\usepackage{url,centernot}
\usepackage{graphicx,color}
\usepackage[color,notref,notcite]{showkeys}
\usepackage{bm}
\usepackage{stmaryrd}
\usepackage{citesort}
\newcommand{\mathbi}[1]{{\boldsymbol #1}}

\def\R{\mathbb{R}}

\def\dist{{\rm dist}}

\def\tpen{R}
\def\iso{{\mathcal L}}

\def\pen{{R}}

\def\n{\mathbf{n}}

\def\dx{\,{\rm d}x}
\def\P{\mathcal{P}}
\def\<{\langle}
\def\>{\rangle}

\def\gr{{\mathcal G}}

\def\diam{\mathrm{diam}}

\def\dsp{\displaystyle} 
\def\div{{\rm div}}
\def\refe#1{\eqref{#1}}

\def\norm#1#2{\Vert#1\Vert_{#2}}

\def\err{\mathsf{err}}

\newcounter{cst}
\def \ctel#1{C_{\refstepcounter{cst}\label{#1}\thecst}}
\def \cter#1{C_{\ref{#1}}}

\def\bt{\begin{theorem}}
\def\et{\end{theorem}}
\def\bl{\begin{lemma}}
\def\el{\end{lemma}}
\def\bc{\begin{corollary}}
\def\ec{\end{corollary}}
\def\bd{\begin{definition}}
\def\ed{\end{definition}}
\def\br{\begin{remark}}
\def\er{\end{remark}}

\newcommand{\x}{\mathbi{x}}
\newcommand{\e}{\mathbi{e}}
\renewcommand{\c}{\mathbi{c}}
\renewcommand{\a}{\mathbi{a}}

\newcommand{\y}{\mathbi{y}}
\newcommand{\z}{\mathbi{z}}
\newcommand{\bu}{\overline{u}}
\newcommand{\bv}{\overline{v}}

\newcommand{\bpsi}{\mathbi{\psi}}

\newcommand{\polyd}{{\mathcal T}}
\newcommand{\mesh}{{\mathcal M}}
\newcommand{\centers}{\mathcal{P}}
\newcommand{\groupcell}{\mathfrak{P}}
\newcommand{\patch}{\mathrm{Pa}}
\newcommand{\Pmesh}{\pi^w_\mesh}
\newcommand{\strip}{\mathbf{S}}

\newcommand{\itr}{$\polyd_0$}
\newcommand{\itrh}{\scalebox{-1}[1]{\itr}}
\newcommand{\itrv}{\scalebox{1}[-1]{\itr}}
\newcommand{\itrr}{\scalebox{-1}[-1]{\itr}}

\def\cell{K}
\newcommand{\edge}{{\sigma}}
\newcommand{\edges}{{\mathcal E}}              
\newcommand{\edgescv}{{{\edges}_\cell}}  
\newcommand{\edgesext}{{{\edges}_{\rm ext}}} 
\newcommand{\edgesint}{{{\edges}_{\rm int}}} 
\newcommand{\centeredge}{\overline{\mathbi{x}}_\edge} 
\newcommand{\vertex}{{\mathsf{v}}}

\newcommand{\be}{\begin{equation}}
\newcommand{\ee}{\end{equation}}
\newcommand{\disc}{{\mathcal D}}

\newcommand{\discs}{{\disc^*}}
\def\tW{\widetilde{W}}

\def\nablacell{\nabla_{\!\cell\,}\!}

\renewcommand{\O}{\Omega}
\def\dr{\partial}
\def\bfn{\mathbf{n}}

\def\ncvedge{\n_{\cell,\edge}}

\renewcommand{\d}{\,{\rm d}}

\newcommand{\ba}{\begin{array}{llll}   }
\newcommand{\bac}{\begin{array}{c}}
\newcommand{\bari}{\begin{array}{r}}
\newcommand{\ea}{\end{array}}

\newcommand{\NORM}[1]{{\left\vert\kern-0.25ex\left\vert\kern-0.25ex\left\vert #1 
    \right\vert\kern-0.25ex\right\vert\kern-0.25ex\right\vert}}

\def\grad{\nabla}

\def\WS{{\rm WS}} 
\def\hI{{\rm I}} 
\def\eaaD{{\rm E}} 

\newtheorem{theorem}{Theorem}[section]
\newtheorem{remark}[theorem]{Remark}

\newtheorem{lemma}[theorem]{Lemma} 
\newtheorem{definition}[theorem]{Definition}
\newtheorem{proposition}[theorem]{Proposition}
\newtheorem{corollary}[theorem]{Corollary}

\numberwithin{equation}{section}

\newcounter{cexp}
\catcode`\@=11
\def\terml#1{T_{\refstepcounter{cexp}\@bsphack
\protected@write\@auxout{}%
           {\string\newlabel{#1}{{\thecexp}{\thepage}}}\thecexp}}
\catcode`\@=12

\begin{document}
\title[Improved $L^2$ estimate for gradient schemes]{Improved $L^2$ estimate for gradient schemes and super-convergence of the TPFA finite volume scheme}

\author{J\'er\^ome Droniou}
\address{School of Mathematical Sciences, Monash University, Clayton, Victoria 3800, Australia.
\texttt{jerome.droniou@monash.edu}}
\author{Neela Nataraj}
\address{Department of Mathematics, Indian Institute of Technology Bombay, Powai, Mumbai 400076, India.
\texttt{neela@math.iitb.ac.in}}
\date{\today}

\maketitle

\begin{abstract}
The gradient discretisation method is a generic framework that is applicable to a number of schemes
for diffusion equations, and provides in particular generic error estimates in $L^2$ and $H^1$-like
norms. In this paper, we establish an improved $L^2$ error estimate for gradient schemes.
This estimate is applied to a family of gradient schemes, namely, the Hybrid Mimetic
Mixed (HMM) schemes, and yields an $\mathcal O(h^2)$ super-convergence rate
in $L^2$ norm, provided local compensations occur
between the cell points used to define the scheme and the centers of mass of the cells.
To establish this result, a modified HMM method is designed by just changing the quadrature
of the source term; this
modified HMM enjoys a super-convergence result even on meshes without local compensations.
Finally, the link between HMM and Two-Point Flux Approximation (TPFA)
finite volume schemes is exploited to partially answer a long-standing conjecture
on the super-convergence of TPFA schemes.
\end{abstract}

{\small

\smallskip

}

\section{Introduction}

When applying a numerical scheme to an elliptic partial differential equation,
the expected rate of convergence is directly dependent on the interpolation
properties of the approximation space. For example, when using a piecewise
constant approximation, as in many finite volume methods, the expected
rate of convergence in $L^2$ norm is $\mathcal O(h)$, where $h$ is the
mesh size. Super-convergence is the phenomenon that occurs when a numerical method displays
a better convergence rate than the expected one.

Let us consider the linear elliptic second-order problem
\be\label{base}
\left\{
\ba
-\div(A\nabla \bu)=f\mbox{ in $\O$},\\
\bu=0\mbox{ on $\dr\O$},
\ea
\right.
\ee
where
\be\label{assump}
\begin{aligned}
&\Omega \subset {\mathbb R}^d \; (d \ge 1)\mbox{ is a bounded domain},\;f\in L^2(\O),\\
&A:\O\rightarrow \mathcal M_d(\R)\mbox{ is measurable, bounded, uniformly elliptic,}\\
&\qquad\mbox{and $A(\x)$ is symmetric for a.e. $\x\in \O$}.
\end{aligned}
\ee
Problem \eqref{base} is understood in the usual weak sense, that  is:
\begin{equation}\label{weak_state}
\mbox{Find $\bu\in H^1_0(\O)$ such that, for all $v\in H^1_0(\O)$},\;
a(\bu,v)=(f,v),
\end{equation}
where $(\cdot,\cdot)$ is the scalar product in $L^2(\O)$ and
\[
a(v,w)=
\int_\O A\nabla v\cdot\nabla {w} \d\x \; \quad \forall v,w \in H^1_0(\O).
\]
Under the assumption \eqref{assump}, Problem \eqref{weak_state} has a unique solution. 

For many (low-order) finite volume methods, $\mathcal O(h)$ error estimates
in the $L^2$ norm and a discrete $H^1$ norm are known, and several super-convergence results
have been numerically observed for the $L^2$ norm without being proved yet (see \cite{review} and references therein).

\medskip

The Two-Point Flux Approximation (TPFA) finite volume scheme is a very popular scheme
used for decades in reservoir simulation \cite{peaceman-history}. It has been fully
analysed in \cite{EGH00}, and extensively tested in a number of situations.
Due to its construction, classical TPFA benchmarks are conducted on 2D meshes made
of acute triangles \cite{BH00,dom-05-fin}. This scheme uses piecewise constant approximations and hence the expected rate of convergence in $L^2$ norm is $\mathcal O(h)$.
However, the aforementioned numerical tests have shown that this piecewise
constant approximation provides an $\mathcal O(h^2)$
estimate of the value of the solution at the circumcenters of the triangles.
Such a super-convergence result was never proved theoretically. 

\medskip 

The main contribution
of this paper is to give a rigorous proof of this super-conver\-gence.
Precisely, we establish an $\mathcal O(h^{2})$
estimate in $L^2$ norm of the difference between
the solution to the TPFA scheme and the piecewise constant projection of
the exact solution constructed from its values at the circumcenters of the triangles.
Our result
covers all 2D meshes encountered in TPFA benchmarking.


\medskip

Previous works have established some relations between
TPFA and $\mathbb{RT}_0$--$\mathbb{P}_0$, provided particular choices of numerical integrations
are used \cite{J84,BMO}. These relations are therefore not exact algebraic
equivalence, and do not allow one to deduce the super-convergence for TPFA
from the super-convergence for $\mathbb{RT}_0$--$\mathbb{P}_0$ \cite{DM85}.
A relation, not based on numerical integration, between TPFA on triangles and $\mathbb{RT}_0$--$\mathbb{P}_0$ mixed finite elements has been established
in \cite{CK}, but has a limited scope, since the source term $f$ must vanish \cite{AMS} (see also
\cite{YMAC,CYA}). If the source term is not zero, then $\mathbb{RT}_0$--$\mathbb{P}_0$ can be reformulated
as a finite volume method, which is different from TPFA since the source term is involved in the definition of the fluxes. We refer to \cite{VW} for a thorough study of mixed finite element methods interpreted as finite volume methods, and related fluxes and properties. 

\medskip

In any case, these various relations between TPFA and $\mathbb{RT}_0$--$\mathbb{P}_0$ do not seem to directly lead to a proof of the observed super-convergence of TPFA. This is due to variations in the choice of approximation points. The TPFA interpretation of $\mathbb{RT}_0$--$\mathbb{P}_0$ requires to introduce new cell unknowns located at the circumcenters of the triangles, which do not correspond to the standard $\mathbb{RT}_0$--$\mathbb{P}_0$ cell unknowns, located at the centers of mass of the triangles,
for which the super-convergence is proved.

\medskip

In \cite{HWY16}, a relationship is established on Voronoi meshes between the TPFA
scheme and a generalised mixed-hybrid mimetic finite difference 
method (with cell and face centers moved away from the centers of mass).
A super-convergence of this method is established, under the assumption
that certain lifting operators exist. This existence is only checked
in the case of rectangular cells (for which TPFA amounts to a finite difference scheme).

\medskip

It should also be mentioned that some post-processing techniques can provide,
under certain circumstances, an $\mathcal O(h^2)$ convergence in $L^2$ norm for functions reconstructed from the
solutions to finite volume approximations. One of these post-processing technique, using two TPFA schemes on two
dual meshes, is described in \cite{OMN11}. These quadratic convergences of post-processed solutions however do not say anything
specific on the super-convergence of the original finite volume scheme.

The super-convergence result for TPFA  established in the present paper holds without post-processing, for the natural unknown at the circumcenter of the triangles, and on all the kinds of triangular meshes used in benchmarking.
 This result therefore appears to solve a long-standing conjecture on this popular finite volume method, on 2D triangular meshes as encountered in practical test-cases.

\medskip

The technique used to prove the super-convergence of TPFA is an indirect one. We use the fact
that, on 2D triangular grids, the TPFA scheme is an HMM method.
HMM sche\-mes, defined in \cite{dro-10-uni}, is a family of methods
that includes mixed-hybrid mimetic finite difference (hMFD) 
schemes \cite{bre-05-fam,BLM:book,mfdrev}, mixed finite volume schemes \cite{dro-06-mix} and 
hybrid finite volume (``SUSHI'') schemes \cite{sushi}. 
The construction of an HMM scheme requires to choose one point inside each
mesh cell. When this point is at the center of mass of the cell, 
HMM schemes boil down to hMFD schemes and super-convergence is then
known \cite{BLM:book,DPLE14,phdbonelle,EB14}. 
But when this cell point is moved away from the center of mass,
super-convergence is less clear and can possibly fail, as we show in a
numerical test. On triangular grids, the TPFA scheme is an HMM method precisely when
these cell points are not located at the centers of mass, but at the
circumcenters of the cells. Establishing the super-convergence of
TPFA through its identification as an HMM scheme therefore requires first
to obtain a super-convergence result for HMM methods with cell points
located away from centers of mass of the cells.

\medskip

This super-convergence for HMM schemes is obtained through a new,
improved $L^2$ estimate for gradient schemes.
A gradient scheme for, say, \eqref{base} is obtained by selecting a family of discrete space and operators, called a gradient discretisation (GD), and by substituting, in the weak formulation of \eqref{base}, the continuous space and operators with these discrete ones. This method is
called the gradient discretisation method (GDM). The vast possible choice
of GD makes the GDM a generic framework for the convergence analysis
of many numerical methods, which include finite elements, mixed finite elements, finite volume, mimetic finite difference methods, HMM, etc. for diffusion, Navier--Stokes, elasticity equations and some other models. We refer to \cite{DEH15} for an analysis of the methods
covered by this framework, and to \cite{eym-12-sma,DL14,dro-14-sto,eym-12-stef,dro-12-gra,zamm2013,dro-14-deg}
for a few models on which the convergence analysis can be carried out within this framework;
see also the monograph \cite{koala} for a complete presentation of the GDM for various boundary conditions. 
Each specific scheme corresponds to a certain choice of GD, and the convergence analysis
conducted in the GDM applies to all choices of GD and thus, to all the schemes covered
by the framework.
A generic error estimate has been established for the GDM applied to \eqref{base}.
This estimate gives the standard $\mathcal O(h)$ rate of convergence in $H^1$ norm for the low-order methods covered by the GDM, such as the HMM schemes \cite{dro-12-gra}.

\medskip 

To summarise, the contributions of this paper are
\begin{itemize}
\item[(i)] an improved $L^2$ estimate for gradient schemes, in any dimension $d$,
\item[(ii)] a modified HMM scheme with unconditional super-convergence, in dimension $d\le 3$,
\item[(iii)] a super-convergence result for HMM, in dimension $d\le 3$, and
\item[(iv)] a super-convergence result for TPFA, in dimension $d=2$ on triangular
meshes as encountered in benchmarks.
\end{itemize} 
The improved $L^2$ error estimate for gradient schemes involves, as in the Aubin--Nitsche trick, the solution to a dual problem.
Applied to HMM schemes, this new estimate
provides an $\mathcal O(h^2)$ super-convergence result when
some form of local compensation occurs; that is, the cell points may
be away from the centers of mass, but not too far away on average
over a few neighbouring cells. The proof of the super-convergence of TPFA then consists in
checking that, for triangular meshes used in TPFA benchmarkings, this local compensation always occur.
A by-product of the proof of super-convergence for HMM schemes is the design of a {\it modified
HMM scheme}, in which only the right-hand side is modified. This modified HMM has
the same matrix, and same computational cost as the original HMM since only the quadrature
of the source term is modified;  but yields super-convergence
for any choice of cell points, even when the standard HMM scheme fails to super-converge.

\medskip

The paper is organised as follows. The description of the TPFA scheme and of the meshes used in benchmarking are 
provided in Subsection \ref{sec:def.TPFA} at the end of this introduction. This section also states our main result, that is the 
super-convergence of TPFA. Section \ref{sec:defGS} recalls the principle of the GDM and Section \ref{sec:improved.L2.GS} establishes the improved $L^2$
estimate. In Section \ref{sec:HMM}, the construction of HMM method is recalled and a modified HMM method is designed.  In Section  \ref{sec:imp.HMM}, we state and prove a new $L^2$ error estimate for HMM, that involves patches of cells. When these patches can be 
chosen so that a compensation occurs, within each patch, between the cell points and the centers of mass, this new 
$L^2$ 
estimate provides the super-convergence of HMM. The proof of the super-convergence of TPFA is given at the end of 
Section \ref{sec:imp.HMM}. Numerical results provided in Section \ref{sec:num.tests} show that in the absence of patches as above, super-convergence may fail for HMM schemes but holds for the modified HMM scheme.
A conclusion, recalling the main results, is given in Section \ref{sec:concl}.
Section \ref{sec:appen}, an appendix, gathers various 
results: a proper analysis of approximate diffusion tensors $A$ in the construction of gradient schemes; some technical 
results used in the rest of the paper; and a discussion on the implementation of the HMM method, the modified HMM 
method, and their corresponding fluxes.

\medskip

Two remarks are in order to conclude
this introduction. First, we consider here homogeneous Dirichlet boundary
conditions in \eqref{base} only for the sake of simplicity. The gradient scheme framework
has been developed for all classical boundary conditions \cite{koala} and our technique applies to other boundary conditions
with minor modifications. Secondly, although we only apply it to HMM and TPFA schemes, the improved $L^2$ error estimate that we establish in the context of the
GDM could certainly lead to super-convergence results for other
schemes contained in this framework, such as discrete duality finite volumes, some
multi-point flux approximation finite volumes, etc.

\subsection{Super-convergence for TPFA}\label{sec:def.TPFA}

Consider a TPFA-admissible mesh $\polyd$ as in \cite{EGH00}. $\polyd$
is therefore a partition
of $\Omega$ into polygonal cells $\mesh$ together with a choice of points $(\x_K)_{K\in\mesh}$
in the cells such that, denoting by $\edges_K$ the edges of $K\in\mesh$,
\begin{itemize}
\item for any neighbourhing cells $K$ and $L$ in $\mesh$,
if $\edge\in \edges_K\cap\edges_L$ then $(\x_K\x_L)\bot \edge$,
\item for any cell $K\in\mesh$, if $\edge\in\edges_K$ and $\edge\subset \partial\Omega$ then
$(\x_K+\R^+\bfn_{K,\edge})\cap \edge\not=\emptyset$,
where $\bfn_{K,\edge}$ is the normal to $\edge$ pointing outward $K$.
\end{itemize}
Let $\edgesint$ be the set of edges interior to $\Omega$ and $\edgesext$ be the
set of edges lying on $\partial\Omega$.
If $A$ is an isotropic tensor, that is, $A(\x)=a(\x){\rm Id}$ for some $a(\x)\in (0,\infty)$,
the TPFA method for \eqref{base} on $\polyd$ reads \cite{EGH00,review}:
\begin{multline}\label{tpfa.scheme}
\mbox{Find $u=(u_K)_{K\in\mesh}$ such that:}\\
\forall K\in\mesh\,,\;\sum_{\edge\in\edgescv\cap\edgesint} \tau_\edge(u_K-u_L)
+\sum_{\edge\in\edgescv\cap\edgesext} \tau_\edge u_K=\int_Kf(\x)\d\x,
\end{multline}
where $L$ is the cell on the other side of $\edge$ if $\edge\in\edgescv\cap\edgesint$,
and, with $a_K$ being the average value of $a$ on $K$,
\begin{align*}
&\forall\edge\in \edgesint\,,\mbox{ if $K\not=L$ are the cells on both
sides of $\edge$,}\;
\tau_\edge=|\edge|\frac{a_Ka_L}{a_Kd_{L,\edge}+a_L d_{K,\edge}}\,,\\
&\forall\edge\in\edgesext\,,\mbox{ if $K$ is the cell such that $\edge\in\edges_K$}\,,\;
\tau_\edge=|\edge|\frac{a_K}{d_{K,\edge}}.
\end{align*}

In 2D, a classical way to construct meshes satisfying the orthogonality property
$(\x_K\x_L)\bot \edge$ is to partition $\O$ into a conforming triangulation
with acute triangles, and to take each $\x_K$ as the circumcenter of $K$.
Three triangulation constructions are widely used in TPFA benchmarkings, see e.g. \cite{BH00,dom-05-fin}:
subdivision, reproduction by symmetry, or reproduction by translation.
Actually, we are not aware of any reported benchmark on TPFA that uses different mesh constructions.

\begin{definition}[Classical TPFA triangulation]\label{def:classical.tpfa.tri}
Let $\O$ be a polygonal bounded open set of $\R^2$.
A classical TPFA triangulation of $\O$ is a conforming acute triangulation $\polyd$
of $\O$ such that, for all $K\in\mesh$, $\x_K$ is the circumcenter of $K$,
and that is constructed in one of the following ways (illustrated for a square domain $\O$ 
in Figure \ref{fig:classical_subd}):
\begin{itemize}
\item \emph{Subdivision}: an initial triangulation \itr{} of $\O$ is chosen, and then subdivided
by creating on each edge an identical number of equally spaced points, by joining the corresponding points on different edges, and by adding the interior points resulting
from intersections of the lines thus created,
\item \emph{Reproduction by symmetry}: an initial triangulation \itr{} of the unit square is chosen,
this unit square is shrunk by a factor $N$ and reproduced in the entire domain by symmetry.
\item \emph{Reproduction by translation}: an initial triangulation \itr{} of the unit square is chosen,
this unit square is shrunk by a factor $N$ and reproduced in the entire domain by translation.
\end{itemize}
\end{definition}

\begin{figure}
\begin{tabular}{ccc}
{\resizebox{0.3\linewidth}{!}{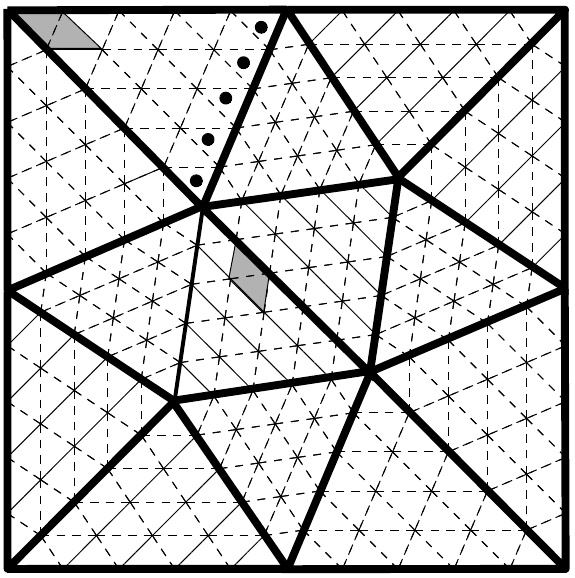}}
&{\resizebox{0.3\linewidth}{!}{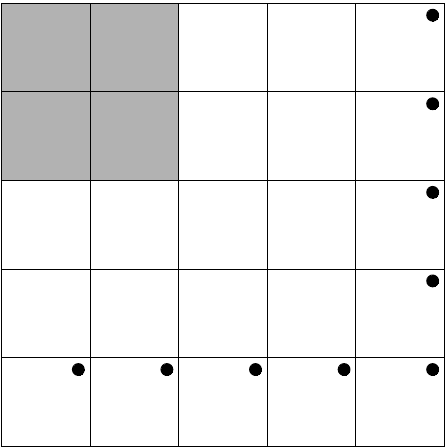}}
&{\resizebox{0.3\linewidth}{!}{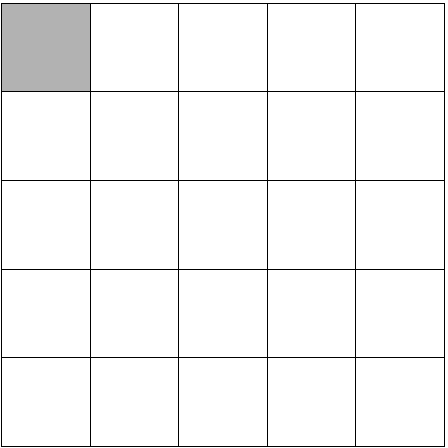}}
\end{tabular}
\caption{Classical acute triangulations: subdivision (left --the initial triangulation
\itr{} is in bold), reproduction by symmetry
(center), and reproduction by translation (right). The greyed and dotted regions are referred to in
the proof, given in Section \ref{sec:imp.HMM}, of Theorem \ref{th:super_cv.tpfa}.}
\label{fig:classical_subd}
\end{figure}

\begin{theorem}[Super-convergence for TPFA on triangles]\label{th:super_cv.tpfa}~\\
Let the assumptions \eqref{assump}, $A=a{\rm Id}$
for some $a:\O\to (0,\infty)$, and that $d=2$ hold.
Also assume that \eqref{base} has the optimal $H^2$ regularity
property (see \eqref{assump.H2}), and that
$f\in H^1(\O)$ and $\bu$ is the solution to
\eqref{weak_state}. Let $\polyd$ be a classical TPFA triangulation of $\O$ in the
sense of Definition \ref{def:classical.tpfa.tri}.
If $u=(u_K)_{K\in\mesh}$ is the solution of the TPFA scheme on $\polyd$ then
there exists $C$, depending only on $\O$, $a$, and \itr{} such that
\be\label{sc-disc-group.tpfa}
\norm{u-\bu_\centers}{L^2(\O)}\le C\norm{f}{H^1(\O)}h_\mesh^{2}.
\ee
Here, $u$ is identified with a piecewise constant function on $\mesh$,
and $\bu_\centers$ is defined by 
\be\label{def.ucenters}
\forall \cell\in\mesh\,,\;\bu_\centers = \bu(\x_\cell)\mbox{ on $\cell$}.
\ee
\end{theorem}

\section{The gradient discretisation method for elliptic PDEs} \label{sec:defGS}

In a nutshell, the gradient discretisation method (GDM) consists in writing a scheme
--called a gradient scheme-- by replacing 
the continuous space and operators by discrete counterparts in the weak formulation of the PDE. These discrete space and
operators are
provided by a gradient discretisation (GD).

\begin{definition}[Gradient discretisation]\label{def:GD}
A gradient discretisation (for homogeneous Dirichlet conditions) is
a triplet $\disc=(X_{\disc,0},\Pi_\disc,\nabla_\disc)$ of:
\begin{itemize}
\item A finite-dimensional space $X_{\disc,0}$ of degrees of freedom, that accounts for the
zero boundary condition,
\item A linear mapping $\Pi_\disc:X_{\disc,0} \rightarrow L^2(\O)$ which reconstructs a function from
the degrees of freedom,
\item A linear mapping $\nabla_\disc:X_{\disc,0} \rightarrow  L^2(\O)^d$ which reconstructs
a gradient from the degrees of freedom. It is chosen such that $\norm{\nabla_\disc\cdot}{L^2(\O)^d}$
is a norm on $X_{\disc,0}$.
\end{itemize}
\end{definition}

Given a gradient discretisation $\disc$, the corresponding  gradient scheme for \eqref{weak_state} is:
\be\label{base.GS}
\mbox{Find $u_\disc\in X_{\disc,0}$ such that, for all $v_\disc\in X_{\disc,0}$, }
a_{\disc}(u_{\disc}, v_{\disc}) =(f,\Pi_\disc v_\disc),
\ee
where 
\[
a_{\disc}(u_{\disc}, v_{\disc}) := \int_\O A\nabla_\disc u_\disc\cdot\nabla_\disc v_\disc\d\x.
\]

The accuracy of a gradient scheme is measured by three
quantities which are defined now.
The first one is a discrete Poincar\'e constant $C_\disc$, which ensures the \emph{coercivity}
of the method.
\be\label{def.CD}
C_\disc := \sup_{w\in X_{\disc,0}}\frac{\norm{\Pi_\disc w}{L^2(\O)}}{\norm{\nabla_\disc w}{L^2(\O)^d}}.
\ee
The second quantity is the interpolation error $S_\disc$, which
measures what is called the \emph{GD-consistency} of the gradient discretisation
(which is known as the interpolation error in the context of finite element methods).
\be\label{def.SD}
\forall \phi\in H^1_0(\O)\,,\;
S_\disc(\phi)=\min_{w\in X_{\disc,0}}\left(\norm{\Pi_\disc w-\phi}{L^2(\O)}
+\norm{\nabla_\disc w-\nabla\phi}{L^2(\O)^d}\right).
\ee
Finally, we measure the \emph{limit-conformity} (or defect of conformity)
of a gradient discretisation through $W_\disc$ defined by
\be\label{def.WD}
\forall \bpsi\in H_{\div}(\O)\,,\;
W_\disc(\bpsi)=\max_{w\in X_{\disc,0}}\frac{\left|  \tW_\disc(\bpsi, w)\right|}{\norm{\nabla_\disc w}{L^2(\O)^d}},
\ee
where $H_{\div}(\O)=\{\bpsi\in L^2(\O)^d\,:\,\div(\bpsi)\in L^2(\O)\}$ and
\be\label{def:tW}
\tW_\disc(\bpsi, w) =\dsp\int_\O \left( \Pi_\disc w \: \div(\bpsi)
+\nabla_\disc w\cdot\bpsi \right) \d\x.
\ee

Using these quantities, the following basic stability and error estimates can be established 
\cite{eym-12-sma,koala}.

\begin{theorem}\label{th:error.est.PDE} Let $\disc$ be a gradient discretisation, $\bu$ be the solution
to \eqref{weak_state} and $u_\disc$  be the solution to \eqref{base.GS}. Then
there exists $C>0$ depending only on $A$ and an upper bound of $C_\disc$ such that
\be\label{est.stability}
\norm{\Pi_\disc u_\disc}{L^2(\O)}+\norm{\nabla_\disc u_\disc}{L^2(\O)^d}
\le C\norm{f}{L^2(\O)}
\ee
and
\be\label{est.error.gs}
\norm{\Pi_\disc u_\disc-\bu}{L^2(\O)}+\norm{\nabla_\disc u_\disc-\nabla \bu}{L^2(\O)^d}
\le C \: \WS_\disc(\bu),
\ee
where 
\be \label{def.ws}
\WS_\disc(\bu)= W_\disc(A\nabla \bu)+S_\disc(\bu).
\ee
\end{theorem}

\begin{remark}[Rates of convergence] For all classical low-order methods based on meshes,
$\mathcal O(h)$ estimates can be obtained, under classical regularity assumptions
on $A$ and $u$, on $W_\disc(A\nabla \bu)$ and $S_\disc(\bu)$ (see \cite{koala}).
The estimate \eqref{est.error.gs} then gives a linear rate of convergence
for these methods.
\label{rem.rates}
\end{remark}

\section{Improved $L^2$ error estimate for gradient schemes}\label{sec:improved.L2.GS}

As mentioned in the introduction, we follow the Aubin--Nitsche idea to establish
an improved $L^2$ error estimate for gradient schemes. We therefore need to define
the adjoint problem of \eqref{weak_state}, and consider its approximation by the GDM.

\medskip
The weak formulation for the dual problem with source term $g \in L^2(\O)$ is:
\begin{equation}\label{weak_adjoint}
\mbox{Find $\varphi_g \in H^1_0(\Omega)$ s.t., for all $w\in H^1_0(\O)$, }
a(w,\varphi_g)=(g,w).
\end{equation}

The gradient scheme corresponding to  \eqref{weak_adjoint} is stated as:
\be\label{adj.GS}
\mbox{Find $\varphi_{g,\disc} \in X_{\disc,0}$ s.t., for all $w_\disc\in X_{\disc,0}$, }
a_{\disc}(w_\disc, \varphi_{g,\disc})=(g, \Pi_\disc w_\disc).
\ee

In order to state the improved $L^2$ error estimate, we need to define
a measure of the interpolation error that, contrary to $S_\disc$, separates
the orders of approximation for the function and its gradient.
If $\alpha>0$, $\phi_\disc\in X_{\disc,0}$ and $\phi\in H^1_0(\O)$, let
\be\label{def.err.alpha}
\hI_{\disc,\alpha}(\phi,\phi_\disc)=\norm{\Pi_\disc \phi_\disc-\phi}{L^2(\O)}
+\alpha \norm{\nabla_\disc \phi_\disc-\nabla\phi}{L^2(\O)^d}.
\ee

\begin{theorem}[Improved $L^2$ error estimate for gradient schemes] 
\label{th-gs-super}  \qquad 

\noindent Assume \eqref{assump}, and  let $\bu$ be the solution to \eqref{weak_state}.
Let $\disc$ be a gradient discretisation in the sense of Definition \ref{def:GD},
and let $u_\disc$ be the solution to the gradient scheme \eqref{base.GS}.
Define
\[
g=\frac{\Pi_\disc u_\disc-\bu}{\norm{\Pi_\disc u_\disc -\bu}{L^2(\O)}}\in L^2(\O)
\]
and let $\varphi_g$ be the solution to \eqref{weak_adjoint}.
Let $\P_\disc:H^2(\O)\cap H^1_0(\O) \to X_{\disc,0}$ be a mapping that selects,
for each $\phi\in H^2(\O)\cap H^1_0(\O)$, an element $\P_\disc\phi\in X_{\disc,0}$.
Then, there exists
$C>0$ depending only on $\O$, $A$ and an upper bound of $C_\disc$ such that
\be\label{improved.gs}
\begin{aligned}
\Vert\Pi_\disc &u_\disc -\bu\Vert_{L^2(\O)}\\
\le{} C &\Big(\left[\alpha^{-1}\hI_{\disc,\alpha}(\bu,\P_\disc \bu)+\WS_\disc(\bu)\right]
\left[\alpha^{-1}\hI_{\disc,\alpha}(\varphi_g,\P_\disc \varphi_g)+\WS_\disc(\varphi_g)\right]\\
&+\hI_{\disc,\alpha}(\bu,\P_\disc \bu)
+\norm{f}{L^2(\O)}\hI_{\disc,\alpha}(\varphi_g,\P_\disc \varphi_g)\\
&+\left|\tW_\disc(A\nabla \bu,\P_\disc\varphi_g)\right|+
\left|\tW_\disc(A\nabla\varphi_g,\P_\disc \bu)\right|\Big),
\end{aligned}
\ee
where $\hI_{\disc,\alpha}$ is defined by \eqref{def.err.alpha},
$\WS_\disc$ is defined by \eqref{def.ws},
and $\tW_\disc$ is defined by \eqref{def:tW}.
\end{theorem}

\begin{remark}[Dominating terms] Following Remark \ref{rem.rates}, 
for low-order methods (with mesh size $h$) it is expected
that $\WS_\disc(\psi)=\mathcal O(h)$
if $\psi\in H^2(\O)$ and $A$ is Lipschitz-continuous. Hence, for a given gradient scheme,
Theorem \ref{th-gs-super} provides a super-convergence result, under the $H^2$ maximal
regularity, if we can find a mapping $\P_\disc$ (usually,
an interpolant) such that
$\hI_{\disc,h}(\psi,\P_\disc\psi)= \mathcal O(h^2)$ and
$\tW_\disc({\bpsi},\P_\disc\psi)=\mathcal O(h^2)$ for all $\psi\in H^2(\O)\cap H^1_0(\O)$
and all ${\bpsi}\in H^1(\O)^d$. This is the strategy
followed in Section \ref{sec:imp.HMM}
to establish super-convergence results for HMM schemes.
\end{remark}

\begin{remark}[$\mathbb{P}_1$ finite elements] 
Conforming and non-conforming $\mathbb{P}_1$ finite elements are gradient schemes \cite{koala,DEH15},
for which the basic estimate \eqref{est.error.gs} provides an $\mathcal O(h)$ rate
of convergence in $L^2$ norm. The improved estimate \eqref{improved.gs}
allows to recover, for these methods, the expected $\mathcal O(h^2)$ rate of
convergence.
\end{remark}




\subsection{Preliminary lemmas}

To prove Theorem \ref{th-gs-super}, we need two technical lemmas.
The first one measures the error committed when replacing the continuous
bilinear form $a(\cdot,\cdot)$ with the discrete bilinear form $a_\disc(\cdot,\cdot)$.

\begin{lemma} 
Let the assumption \eqref{assump} hold and let
$\psi, \phi \in H^1_0(\O)$ be such that $\div(A\nabla\psi)\in L^2(\O)$
and $\div(A\nabla\phi)\in L^2(\O)$.
Then, for any $\psi_\disc,\phi_\disc\in X_{\disc,0}$, it holds true that: 
\begin{align} \label{a_estimate}
|a(\psi, \phi) - a_\disc(\psi_\disc, \phi_\disc)| \le  \eaaD_\disc(\psi,\phi,\psi_\disc,\phi_\disc),
 \end{align}
 where 
\be\label{def.ED}
 \begin{aligned} 
   \eaaD_\disc(\psi,\phi,\psi_\disc,\phi_\disc)  ={}&  |\tW_\disc(A\nabla \psi, \phi_\disc)| +
 |\tW_\disc(A\nabla \phi, \psi_\disc)|  \\
&+ \norm{\div(A \nabla \phi)}{L^2(\O)} \hI_{\disc,\alpha}(\psi,\psi_\disc)
 \\
 &  + \norm{\div (A \nabla \psi)}{L^2(\O)} \hI_{\disc,\alpha}(\phi,\phi_\disc)  \\
& +
  \norm{A}{\infty}\alpha^{-2}\hI_{\disc,\alpha}(\psi,\psi_\disc) \hI_{\disc,\alpha}(\phi,\phi_\disc).  \\
  \end{aligned}
\ee
 \end{lemma}
 \begin{proof} 
 Let 
 \begin{align*}
 T & = a(\psi, \phi) - a_\disc(\psi_\disc, \phi_\disc) \\
 & = \int_\O A \nabla \psi \cdot \nabla \phi \d\x - \int_\O A \nabla_\disc \psi_\disc \cdot \nabla_\disc \phi_\disc  \d\x.
 \end{align*}
 Introduce  $\nabla_\disc \psi_\disc$ in the first term in the above integral to obtain 
 \begin{align*}
 T & =\int_\O A (\nabla \psi-\nabla_\disc \psi_\disc ) \cdot \nabla \phi \d\x + 
\int_\O A \nabla_\disc \psi_\disc \cdot (\nabla \phi -\nabla_\disc\phi_\disc) \d\x.
 \end{align*}
 Now introduce $\nabla \psi$ in the second term to obtain
 \begin{align} 
 T ={}&\int_\O A (\nabla \psi-\nabla_\disc \psi_\disc) \cdot \nabla \phi \d\x\nonumber\\
& + \int_\O A (\nabla_\disc \psi_\disc - \nabla \psi) \cdot (\nabla \phi -\nabla_\disc\phi_\disc) \d\x \nonumber\\
 & +\int_\O A  \nabla \psi \cdot (\nabla \phi -\nabla_\disc \phi_\disc ) \d\x=: T_1 + T_2 + T_3.
\label{a_T}
 \end{align}
 The term $T_1$ is re-written as
 \begin{align} 
 T_1  & = \int_\O A \nabla \psi \cdot \nabla \phi \d\x  - \int_\O A \nabla_\disc \psi_\disc 
\cdot \nabla \phi \d\x  \nonumber \\
 & = - \int_\O \psi \div(A \nabla \phi) \d\x - \tW_\disc(A \nabla \phi, \psi_\disc) 
 + \int_\O  \div(A \nabla \phi)  \Pi_\disc \psi_\disc \d\x, \label{T1.sym}
 \end{align}
and this leads to
 \begin{align} \label{a_T1}
 |T_1| & \le |\tW_\disc(A \nabla \phi, \psi_\disc)|  + \norm{\div(A \nabla \phi)}{L^2(\O)}
 \norm{\psi - \Pi_\disc \psi_\disc}{L^2(\O)} \nonumber \\
 & \le |\tW_\disc(A \nabla \phi, \psi_\disc)| + \norm{\div(A \nabla \phi)}{L^2(\O)} \hI_{\disc,\alpha}(\psi,\psi_\disc).
 \end{align}
 The term $T_2$ is estimated as 
 \begin{align} \label{a_T2}
 |T_2| & \le \norm{A}{\infty} \norm{\nabla \psi - \nabla_\disc \psi_\disc}{L^2(\O)^d}
 \norm{\nabla \phi - \nabla_\disc \phi_\disc}{L^2(\O)^d} \nonumber \\
 & \le\norm{A}{\infty}\alpha^{-2}\hI_{\disc,\alpha}(\psi,\psi_\disc) \hI_{\disc,\alpha}(\phi,\phi_\disc).
  \end{align}
  The term $T_3$ is estimated exactly as $T_1$ interchanging the roles of $(\psi, \psi_\disc)$ and 
$(\phi,\phi_\disc)$, which leads to
   \begin{align} \label{a_T3}
 |T_3| & \le |\tW_\disc(A \nabla \psi, \phi_\disc)|  + \norm{\div(A \nabla \psi)}{L^2(\O)}
  \hI_{\disc,\alpha}(\phi,\phi_\disc).
 \end{align}
  A substitution of \eqref{a_T1}-\eqref{a_T3} in  \eqref{a_T} yields the required estimate in 
  \eqref{a_estimate}.
 \end{proof}

The following trivial lemma will enable us to evaluate the distance
between $P_\disc \bu$ and $u_\disc$ (and similar for $\varphi_g$) in the proof of Theorem
\ref{th-gs-super}.

 \begin{lemma} \label{intermediate}
Under the assumption \eqref{assump}, let $\bu$  be the solution to \eqref{weak_state}.
Let $\disc$ be a gradient discretisation in the sense of Definition \ref{def:GD},
and denote the solution to the corresponding gradient scheme \eqref{base.GS}  by $u_\disc$.
Then, there exists $C>0$ depending only on $A$ and an upper bound of $C_\disc$
such that, for all $v_\disc\in X_{\disc,0}$,
\be\label{interm.Pi}
\norm{\Pi_\disc v_\disc - \Pi_\disc u_\disc}{L^2(\O)}
 \le  \hI_{\disc,\alpha}(\bu,v_\disc)+C \, \WS_\disc(\bu),
\ee
and
\be\label{interm.grad}
\norm{\nabla_\disc v_\disc - \nabla_\disc u_\disc}{L^2(\O)^d}\\
 \le \alpha^{-1} \hI_{\disc,\alpha}(\bu,v_\disc)+C \, \WS_\disc(\bu).
\ee
\end{lemma}

\begin{remark} Estimate \eqref{interm.Pi} will not be used in the sequel, but
is stated for the sake of completeness since it is required when dealing with PDEs with
lower order terms such as $-\div(A\nabla \bu)+\bu=f$.
\end{remark}

\begin{proof} A use of triangle inequality after introducing $\bu$ as an intermediate term, along with the definition of $\hI_{\disc,\alpha}(\bu,v_\disc)$ and the estimate \eqref{est.error.gs} in Theorem \ref{th:error.est.PDE}, yields \eqref{interm.Pi}. Similarly, \eqref{interm.grad} can be established by introducing $\nabla \bu$ as an intermediate term. 
\end{proof}

\subsection{Proof of the improved $L^2$ estimate}

We now turn to the proof of Theorem \ref{th-gs-super}.

From \eqref{weak_adjoint} with $w=\bu$ and \eqref{adj.GS} with $w_\disc=u_\disc$,
\be \label{proof.1}
(g, \bu - \Pi_\disc u_\disc)  =a(\bu, \varphi_g) - a_\disc(u_\disc, \varphi_{g,\disc}).
\ee
Since $\div(A\nabla \bu)=-f\in L^2(\O)$ and $\div(A\nabla\varphi_g)=-g\in L^2(\O)$,
a use of \eqref{a_estimate} in \eqref{proof.1} leads to 
\begin{align}
\Vert \bu-\Pi_\disc u_\disc\Vert_{L^2(\O)}
&=(g, \bu - \Pi_\disc u_\disc) \nonumber\\
& \le a_\disc(\P_\disc \bu, \P_\disc \varphi_g) - a_\disc(u_\disc, \varphi_{g,\disc}) +  \eaaD_\disc(\bu,\varphi_g,\P_\disc \bu,\P_\disc \varphi_g) \nonumber \\
 & = a_\disc(\P_\disc \bu - u_\disc, \P_\disc \varphi_g - \varphi_{g,\disc}) +
  a_\disc( u_\disc, \P_\disc \varphi_g - \varphi_{g,\disc}) \nonumber \\
  & \quad +  a_\disc(\P_\disc \bu - u_\disc,  \varphi_{g,\disc})  +
  \eaaD_\disc(\bu,\varphi_g,\P_\disc \bu,\P_\disc \varphi_g)  \nonumber \\
 & =: T_1 + T_2 + T_3 +  \eaaD_\disc(\bu,\varphi_g,\P_\disc \bu,\P_\disc \varphi_g).
 \label{proof.2}
\end{align}
In the rest of this proof, we denote $\mathcal A\lesssim \mathcal B$ for
$\mathcal A\le C\mathcal B$ with $C$ depending only on $\O$, $A$ and an upper
bound of $C_\disc$.
Using the Cauchy--Schwarz inequality and \eqref{interm.grad},
the term $T_1$ can be estimated as
\begin{align}
|T_1| \lesssim{}& \norm{\nabla_\disc \P_\disc \bu-\nabla_\disc u_\disc}{L^2(\O)^d}
\norm{\nabla_\disc \P_\disc \varphi_g-\nabla_\disc \varphi_{g,\disc}}{L^2(\O)^d}\nonumber\\
\lesssim{}& \left[\alpha^{-1}\hI_{\disc,\alpha}(\bu,\P_\disc \bu)+\WS_\disc(\bu)\right]
\left[\alpha^{-1}\hI_{\disc,\alpha}(\varphi_g,\P_\disc \varphi_g)+\WS_\disc(\varphi_g)\right].
 \label{T1}
\end{align}
Consider the term $T_2$ now. Simple manipulations and a use of \eqref{adj.GS} lead to 
\begin{align} 
T_2  ={}& a_\disc(u_\disc, \P_\disc \varphi_g) -  a_\disc(u_\disc,  \varphi_{g,\disc}) \nonumber \\
 ={}& - a_\disc(\P_\disc \bu- u_\disc, \P_\disc \varphi_g) + a_\disc(\P_\disc \bu-u_\disc,  \varphi_{g,\disc})
+ a_\disc(\P_\disc \bu, \P_\disc \varphi_g-\varphi_{g,\disc}) \nonumber \\
 ={}& -\big[a_\disc(\P_\disc \bu- u_\disc, \P_\disc \varphi_g) -
(g, \Pi_\disc(\P_\disc \bu - u_\disc))\big]   + a_\disc(\P_\disc \bu, \P_\disc \varphi_g -\varphi_{g,\disc})\nonumber\\
={}&-T_{2,1}+T_{2,2}.
\label{T2}
\end{align}
Since $-\div(A\nabla\varphi_g)=g$, we write
\begin{align*}
T_{2,1}={}&\int_\Omega \left[ \nabla_\disc (\P_\disc \bu- u_\disc) \cdot A\nabla \varphi_g
- g \Pi_\disc(\P_\disc \bu-u_\disc)\right] \d\x\\
&+ 
\int_\O A \nabla_\disc (\P_\disc \bu- u_\disc) \cdot (\nabla_\disc \P_\disc \varphi_g - \nabla\varphi_g) \d\x\\
={}&\tW_\disc(A\nabla\varphi_g,\P_\disc \bu-u_\disc)
+\int_\O A \nabla_\disc (\P_\disc \bu- u_\disc) \cdot (\nabla_\disc \P_\disc \varphi_g - \nabla\varphi_g) \d\x.
\end{align*}
We then use \eqref{def.WD} and \eqref{interm.grad} and
the definition of $\hI_{\disc,\alpha}(\varphi_g,\P_\disc\varphi_g)$ to obtain
\begin{align}
|T_{2,1}|\lesssim{}& W_\disc(A\nabla\varphi_g)\norm{\nabla_\disc \P_\disc \bu-\nabla_\disc u_\disc}{L^2(\O)^d}\nonumber\\
&+\norm{\nabla_\disc \P_\disc \bu-\nabla_\disc u_\disc}{L^2(\O)^d}
\norm{\nabla_\disc \P_\disc \varphi_g-\nabla \varphi_g}{L^2(\O)^d}\nonumber\\
\lesssim{}& \left[\alpha^{-1}\hI_{\disc,\alpha}(\varphi_g,\P_\disc\varphi_g)+
\WS_\disc(\varphi_g)\right]
\left[\alpha^{-1}\hI_{\disc,\alpha}(\bu,\P_\disc \bu)+\WS_\disc(\bu)\right].
\label{T21}
\end{align}
{}From \eqref{a_estimate}, \eqref{weak_adjoint} and \eqref{adj.GS}, the term $T_{2,2}$
can be estimated as follows:
\begin{align} \label{T22}
|T_{2,2}| &\le |a(\bu,\varphi_g) -a_\disc(\P_\disc \bu , \varphi_{g,\disc}) |
+  \eaaD_\disc(\bu,\varphi_g,\P_\disc \bu,\P_\disc\varphi_g)  \nonumber \\
 &\le  |(\bu - \Pi_\disc \P_\disc \bu,g )| +  \eaaD_\disc(\bu,\varphi_g,\P_\disc \bu,\P_\disc\varphi_g) \nonumber \\
 & \le \norm{g}{L^2(\O)}\hI_{\disc,\alpha}(\bu,\P_\disc \bu)+ \eaaD_\disc(\bu,\varphi_g,\P_\disc \bu,\P_\disc\varphi_g).
\end{align}
A substitution of \eqref{T21} and \eqref{T22} into \eqref{T2} leads to an estimate for $T_2$:
\begin{align}
|T_2|\lesssim{}&\left[\alpha^{-1}\hI_{\disc,\alpha}(\varphi_g,\P_\disc\varphi_g)+
\WS_\disc(\varphi_g)\right]
\left[\alpha^{-1}\hI_{\disc,\alpha}(\bu,\P_\disc \bu)+\WS_\disc(\bu)\right]\nonumber\\
&+\norm{g}{L^2(\O)}\hI_{\disc,\alpha}(\bu,\P_\disc \bu)+ \eaaD_\disc(\bu,\varphi_g,\P_\disc \bu,\P_\disc\varphi_g).
\label{est.T2}
\end{align}
The term $T_3$ is similar to $T_2$, upon swapping the primal and dual problems (both continuous
and discrete), that is $(f,\bu,u_\disc,g,\varphi_g,\varphi_{g,\disc})
\leftrightarrow (g,\varphi_g,\varphi_{g,\disc},f,\bu,u_\disc)$. Hence, with these substitutions in
\eqref{est.T2}, we see that 
\begin{align}
|T_3|\lesssim{}&\left[\alpha^{-1}\hI_{\disc,\alpha}(\varphi_g,\P_\disc\varphi_g)+
\WS_\disc(\varphi_g)\right]
\left[\alpha^{-1}\hI_{\disc,\alpha}(\bu,\P_\disc \bu)+\WS_\disc(\bu)\right]\nonumber\\
&+\norm{f}{L^2(\O)}\hI_{\disc,\alpha}(\varphi_g,\P_\disc \varphi_g)+ \eaaD_\disc(\bu,\varphi_g,\P_\disc \bu,\P_\disc\varphi_g).
\label{est.T3}
\end{align}
Since $\norm{g}{L^2(\O)}=1$, a substitution in \eqref{proof.2} of the estimates \eqref{T1},
\eqref{est.T2} and \eqref{est.T3} for $T_1$, $T_2$ and $T_3$
leads to
\begin{align*}
\Vert \bu-&\Pi_\disc u_\disc\Vert_{L^2(\O)}\\
\lesssim{}& \left[\alpha^{-1}\hI_{\disc,\alpha}(\bu,\P_\disc \bu)+\WS_\disc(\bu)\right] \left[\alpha^{-1}\hI_{\disc,\alpha}(\varphi_g,\P_\disc\varphi_g)+
\WS_\disc(\varphi_g)\right]\\
&+\hI_{\disc,\alpha}(\bu,\P_\disc \bu)+\norm{f}{L^2(\O)}\hI_{\disc,\alpha}(\varphi_g,\P_\disc \varphi_g)+
\eaaD_\disc(\bu,\varphi_g,\P_\disc \bu,\P_\disc\varphi_g).
\end{align*}
The proof of Theorem \ref{th-gs-super}
is complete by recalling the definition \eqref{def.ED} of $\eaaD_\disc$,
by using
\begin{multline*}
\alpha^{-2}\hI_{\disc,\alpha}(\bu,\P_\disc \bu)
\hI_{\disc,\alpha}(\varphi_g,\P_\disc\varphi_g)\\
\le \left[\alpha^{-1}\hI_{\disc,\alpha}(\bu,\P_\disc \bu)+\WS_\disc(\bu)\right]
\left[\alpha^{-1}\hI_{\disc,\alpha}(\varphi_g,\P_\disc\varphi_g)+
\WS_\disc(\varphi_g)\right],
\end{multline*}
and by noticing that $-\div(A\nabla\varphi_g)=g$ (that has $L^2(\O)$ norm equal to $1$)
and that $-\div(A\nabla \bu)=f$.

\begin{remark}
The symmetry of $A$, assumed in \eqref{assump}, is not a restrictive hypothesis
in applications. However, Theorem \ref{th-gs-super} and all
subsequent results in this paper are valid if $A$ is not symmetric. In some terms,
$A$ simply needs to be replaced with $A^T$: in $\tW(A\nabla\varphi_g,\P_\disc\bu)$
in \eqref{improved.gs}; in $\tW(A\nabla\phi,\psi_\disc)$
and $\div(A\nabla\phi)$ in \eqref{def.ED} and \eqref{T1.sym}; etc.
\end{remark}

\section{Super-convergence for a modified HMM scheme}\label{sec:HMM}

Here, we recall some notations and the definition of the HMM scheme, which is
a gradient scheme for a specific gradient discretisation.
Then, we design a \emph{modified HMM scheme} with a better quadrature rule for the source term
and we prove that this modified scheme super-converges, without any assumption on the mesh or
regularity assumption on $f$. In the next section, we use this super-convergence of
the modified HMM scheme to obtain, under the assumption that $f\in H^1(\O)$, a super-convergence for HMM schemes in the case where, on average on patches of cells, the ``cell points'' $\centers$
of Definition \ref{def:polymesh} are not far from the centers of mass of the cells.

\subsection{Polytopal meshes and definition of the HMM gradient discretisation}\label{sec:poly.def.hmm}

Let us recall the definition of the gradient discretisations
that correspond to HMM schemes, starting with the definition
of a polytopal mesh (we follow \cite{DEH15}, without including the vertices which are
not useful to our purpose).

\begin{definition}[Polytopal mesh]\label{def:polymesh}~
Let $\Omega$ be a bounded polytopal open subset of $\R^d$ ($d\ge 1$). 
A polytopal mesh of $\O$ is $\polyd = (\mesh,\edges,\centers)$, where:
\begin{enumerate}
\item $\mesh$ is a finite family of non empty connected polytopal open disjoint subsets of $\O$ (the cells) such that $\overline{\O}= \dsp{\cup_{\cell \in \mesh} \overline{\cell}}$.
For any $\cell\in\mesh$, $|\cell|>0$ is the measure of $\cell$ and $h_\cell$ denotes the diameter of $\cell$.

\item $\edges$ is a finite family of disjoint subsets of $\overline{\O}$ (the edges of the mesh in 2D,
the faces in 3D), such that any $\edge\in\edges$ is a non empty open subset of a hyperplane of $\R^d$ and $\edge\subset \overline{\O}$.
Assume that for all $\cell \in \mesh$ there exists  a subset $\edgescv$ of $\edges$
such that the boundary of $K$ is ${\bigcup_{\edge \in \edgescv}} \overline{\edge}$. 
We then denote by $\mesh_\edge = \{\cell\in\mesh\,:\,\edge\in\edgescv\}$, 
and assume that, for all $\edge\in\edges$, $\mesh_\edge$ has exactly one element
and $\edge\subset\partial\O$, or $\mesh_\edge$ has two elements and
$\edge\subset\O$. 
Let $\edgesint$ be the set of all interior faces, i.e. $\edge\in\edges$ such that $\edge\subset \O$, and $\edgesext$ the set of boundary
faces, i.e. $\edge\in\edges$ such that $\edge\subset \dr\O$.
The $(d-1)$-dimensional measure of $\edge\in\edges$ is $|\edge|$,
and its center of mass is $\centeredge$.

\item $\centers = (\x_\cell)_{\cell \in \mesh}$ is a family of points of $\O$ indexed by $\mesh$ and such that, for all  $\cell\in\mesh$,  $\x_\cell\in \cell$.
Assume that any cell $\cell\in\mesh$ is strictly $\x_\cell$-star-shaped, meaning that 
if $\x\in\overline{\cell}$ then the line segment $[\x_\cell,\x)$ is included in $\cell$.

\end{enumerate}
\end{definition}

For all $\cell \in \mesh$, denote the center
of mass of $K$ by $\overline{\x}_\cell$ and, if $\edge \in \edgescv$, denote the (constant) unit vector normal to $\edge$ outward to $\cell$ by $\bfn_{K,\edge}$.
Also, let ${\rm d}_{\cell,\edge}$ be the signed orthogonal
distance between $\x_\cell$ and $\edge$ (see Fig. \ref{fig.dksigma}), that is: 
\begin{equation}
 {\rm d}_{\cell,\edge} = (\x - \x_\cell) \cdot \ncvedge\,\quad\forall \x \in \edge \label{def.dcvedge}
\end{equation}
(note that $(\x - \x_\cell) \cdot \ncvedge$ is constant for $\x \in \edge$).
The fact that $\cell$ is strictly star-shaped with respect to $\x_\cell$ is equivalent
to ${\rm d}_{\cell,\edge}> 0$ for all $\edge\in\edgescv$.
For all $\cell\in\mesh$ and $\edge\in\edgescv$, we denote by $D_{\cell,\edge}$ the cone with apex $\x_\cell$ and base $\edge$, that is $D_{\cell,\edge}=\{ t \x_\cell +(1-t) \y\,:\, t\in (0,1),\,
\y\in \edge\}$.

The size of the discretisation is $h_\mesh=\sup\{h_\cell\,:\; \cell\in \mesh\}$ and the
regularity factor is
\be\label{def:theta}
\theta_\polyd = \max_{\edge\in\edgesint\,,\;\mesh_\edge=\{\cell,\cell'\}} \frac {{\rm d}_{\cell,\edge}} {{\rm d}_{\cell',\edge}}
+ \max_{\cell\in\mesh}
\left(\max_{\edge\in\edgescv}\frac{h_\cell}{d_{\cell,\edge}}+{\rm Card}(\edgescv)\right).
\ee
An upper bound on $\theta_\polyd$ imposes three geometrical conditions:
the orthogonal distance $d_{K,\edge}$ between $\x_K$ and $\edge\in\edgescv$ must be comparable
to the diameter $h_K$ of $K$; the orthogonal distances $d_{K,\edge}$ and $d_{K',\edge}$
between $\sigma$ and
its two neighbouring cell points $\x_K$ and $\x_{K'}$ must have similar magnitudes
(see Figure \ref{fig.dksigma});
and there is a global upper bound on the number of faces of each
cell.

\begin{figure}[htb]
\begin{center}
\resizebox{.5\textwidth}{!}{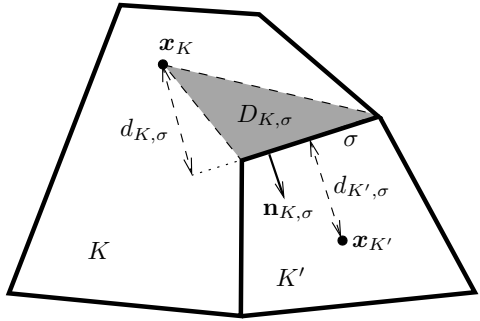}
\caption{A cell $\cell$ of a polytopal mesh.}
\label{fig.dksigma}
\end{center}
\end{figure}

\begin{definition}[HMM gradient discretisation]\label{def.HMM.GD}
Let $\polyd$ be a polytopal mesh of $\O$ as per Definition \ref{def:polymesh}.
An HMM gradient discretisation $\disc=(X_{\disc,0},\Pi_\disc,\nabla_\disc)$ is
defined by:
\begin{enumerate}
\item $X_{\disc,0}=\R^{\mesh}\times \R^{\edgesint}\times \{0\}^{\edgesext}$ is the space of degrees of freedom in the cell and on the interior faces of the
mesh:
\[
\qquad X_{\disc,0}=\{v=((v_\cell)_{\cell\in\mesh},(v_\edge)_{\edge\in\edges})\,:\,
v_\cell\in \R\,,\;v_\edge\in\R\,,\; v_\edge=0\mbox{ if $\edge\in\edgesext$}\}.
\]

\item $\Pi_\disc:X_{\disc,0}\to L^2(\O)$ is the following piecewise constant reconstruction on the 
mesh:
\[
\forall v\in X_{\disc,0}\,,\;\forall \cell\in\mesh\,,\;\Pi_\disc v=v_\cell\mbox{ on $\cell$}.
\]

\item $\nabla_\disc:X_{\disc,0}\to L^2(\O)^d$ reconstructs piecewise constant gradients on
the cones $(D_{\cell,\edge})_{\cell\in\mesh\,,\;\edge\in\edgescv}$:
\begin{align}
\forall v\in X_{\disc,0}&\,,\;\forall \cell\in\mesh\,,\;\forall \edge\in\edgescv\,,\nonumber\\
&\nabla_\disc v = \nablacell v + \frac{\sqrt{d}}{d_{\cell,\edge}}[\iso_\cell\pen_\cell(v)]_\edge\bfn_{\cell,\edge}\mbox{ on $D_{\cell,\edge}$},
\label{def:nablaD.HMM}
\end{align}
where:
\begin{itemize}
\item $\dsp \nablacell v=\frac{1}{|\cell|}\sum_{\edge\in\edgescv}|\edge|v_\edge \bfn_{\cell,\edge}$,
\item $\pen_\cell:X_{\disc,0}\rightarrow \R^{\edgescv}$ is given
by $\tpen_\cell(v)=(\tpen_{\cell,\edge}(v))_{\edge\in\edgescv}$ with
$\tpen_{\cell,\edge}(v)=v_\edge-v_\cell-\nablacell v\cdot(\centeredge-\x_\cell)$,
\item $\iso_\cell$ is an isomorphism of the space ${\rm Im}(\tpen_\cell)$.
\end{itemize}
\end{enumerate}
\end{definition}
If $\disc$ is an HMM gradient discretisation,
define $\zeta_\disc$ as the smallest positive number such that, for all $\cell\in\mesh$ and all
$v\in X_{\disc,0}$,
\be\label{def:zetaD}
\begin{aligned}
\zeta_\disc^{-1}\sum_{\edge\in\edgescv}|D_{\cell,\edge}|\left|\frac{\tpen_{\cell,\edge}(v)}{d_{\cell,\edge}}\right|^2
\le \sum_{\edge\in\edgescv}& |D_{\cell,\edge}|\left|\frac{[\iso_\cell\tpen_\cell(v)]_\edge}{d_{\cell,\edge}}\right|^2\\
&\le\zeta_\disc\sum_{\edge\in\edgescv}|D_{\cell,\edge}|\left|\frac{\tpen_{\cell,\edge}(v)}{d_{\cell,\edge}}\right|^2.
\end{aligned}
\ee

It is proved in \cite{dro-12-gra} that, if $A$ is piecewise constant on the mesh,
for any HMM scheme $\mathcal S$ (as defined in \cite{dro-10-uni})
on a polytopal mesh $\polyd$, there exists a choice
of isomorphisms $(\iso_\cell)_{\cell\in\mesh}$
such that $\mathcal S$ is the gradient scheme corresponding to the gradient
discretisation $\disc$ given by Definition \ref{def.HMM.GD}.

\begin{remark} If $A$ is not piecewise constant, the HMM method corresponds to a
modified gradient scheme which consists in writing \eqref{base.GS} with $A$
replaced by its $L^2$ projection on piecewise constant matrix-valued functions on $\mesh$. As shown in Section
\ref{sec:appen.modified} in the appendix, this modification of the gradient scheme
preserves the basic rates of convergence in Theorem \ref{th:error.est.PDE}, as
well as the super-convergence results established below (see Theorems \ref{th.supercv.discs}
and \ref{th-supercv-HMM-group}).
In the following, we will therefore only consider the standard gradient scheme
\eqref{base.GS}, having ensured that the results established for this one also apply
to the original HMM methods even if $A$ is not piecewise constant (which would be
an unreasonable constraint given the assumption \eqref{assump.H2} to come).
\end{remark}

As for conforming and non-conforming finite element methods,
the super-convergen\-ce estimate for HMM schemes requires a higher
regularity of the solution. In particular, a convergence rate of order 2 is obtained 
 under the following
$H^2$ regularity assumption (which holds if $A$ is Lipschitz continuous and $\O$ is convex).
\be\label{assump.H2}
\begin{aligned}
&\mbox{For all $f\in L^2(\O)$, the solution $\bu$ to \eqref{base} belongs to $H^2(\O)$}\\
&\mbox{and $\norm{\bu}{H^2(\O)}+\norm{A\nabla\bu}{H^1(\O)^d}\le C\norm{f}{L^2(\O)}$,}\\
&\mbox{with $C$ depending only on $\O$ and $A$.}
\end{aligned}
\ee
Under this assumption, the solution $\varphi_g$ to \eqref{weak_adjoint} also satisfies
the $H^2$ regularity (with $f$ replaced with $g$ in the estimate).

\begin{remark}[$H^{1+\delta}$ super-convergence]
If we only assume an $H^{1+\delta}$ regularity property instead of \eqref{assump.H2}, then
standard interpolation techniques can be applied in the proofs below, and we still
obtain some super-convergence results (albeit of reduced order, typically
$h_\mesh^{2\delta}$ instead of $h_\mesh^2$, and possibly by strengthening the
regularity assumptions on $f$ if $\delta$ is small).
\end{remark}

In the rest of this section, we use the following notation:
\be\label{notation.lesssim.hMFD}
\begin{aligned}
&\mbox{$\mathcal A \lesssim \mathcal B$ means that 
$\mathcal A\le C\mathcal B$ with $C$ depending only on}\\
&\mbox{$\O$, $A$ and an upper bound
of $\theta_\polyd+\zeta_\disc$.}
\end{aligned}
\ee

\subsection{A modified HMM scheme with better source term approximation}\label{sec:mod.hmm}

We define here the modified HMM scheme and prove its super-convergence. 

\begin{definition}[Modified HMM gradient discretisation]
Let $\disc=(X_{\disc,0},\Pi_{\disc},\nabla_\disc)$ be an HMM gradient
discretisation in the sense of Definition \ref{def.HMM.GD}.
The modified HMM gradient discretisation is $\discs=(X_{\disc,0},\Pi_{\discs},\nabla_\disc)$, where the reconstruction $\Pi_\discs$ is defined by
\be\label{def.Pidiscs}
\forall v\in X_{\disc,0}\,,\;\forall \cell\in\mesh\,,\;\forall \x\in \cell\,,\\
\Pi_{\discs}v(\x)=\Pi_\disc v(\x) + \nablacell v\cdot(\x-\x_\cell).
\ee
\end{definition}

We notice that using $\discs$ in the gradient scheme \eqref{base.GS} only modifies
the discretisation of the source term, and not the scheme's matrix. This modified
HMM scheme is therefore only marginally more expensive than a standard HMM
scheme. Moreover, it enjoys a better super-convergence
result than the HMM method, since this super-convergence (i) does not require $\centers$ in Definition \ref{def:polymesh} to be close on average to the centers of mass
of the cells, (ii) does not require the $H^1$
regularity of $f$, and (iii) gives an $\mathcal O(h_\mesh^2)$ approximation
of the solution $ \bu$ to \eqref{weak_state}
rather than its piecewise constant projection. The influence on super-convergence of the choice of quadrature
rule for the source term was already noticed in \cite{OMN09} for the TPFA scheme in dimension 1.

\begin{theorem}[Super-convergence for the modified HMM method]\label{th.supercv.discs}
Assume \eqref{assump}, \eqref{assump.H2}, and that $d\le 3$. Let $\bu$ be the solution to
\eqref{weak_state}. Take $\polyd$ as a polytopal mesh in the sense of
Definition \ref{def:polymesh}. Let $u_{\discs}$ be the solution of the gradient scheme
\eqref{base.GS} for the modified HMM gradient discretisation $\disc=\discs$ defined above.
Then, recalling the notation \eqref{notation.lesssim.hMFD},
\be\label{sc-discs}
\norm{\Pi_{\discs} u_{\discs}-\bu}{L^2(\O)}\lesssim\norm{f}{L^2(\O)} h_\mesh^2
\ee
and
\be\label{cvgrad-discs}
\norm{\nabla_{\disc} u_{\discs}-\nabla \bu}{L^2(\O)}\lesssim\norm{f}{L^2(\O)}h_\mesh.
\ee
\end{theorem}

Let us recall a few results that will be used in the proof of this theorem.
From \cite[Lemma 12.4]{koala}:
\be\label{average.grad}
\forall v\in X_{\disc,0}\,,\;
\forall \cell\in\mesh\,,\;\int_\cell \nabla_\disc v(\x)\d\x=|\cell|\nablacell v.
\ee
As a consequence,
\be\label{est.grad}
|\nablacell v|\le |\cell|^{-\frac{1}{2}}\norm{\nabla_\disc v}{L^2(\cell)}.
\ee
By \cite[Propositions 12.14 and 12.15]{koala},
\begin{align}
\label{est.CD}
&C_\disc\lesssim 1\,,\\
\label{est.SD}
&\forall \phi\in H^2(\O)\cap H^1_0(\O)\,,\; S_\disc(\phi)\lesssim\norm{\phi}{H^2(\O)}h_\mesh\,,\\
\label{est.WD}
&\forall\bpsi\in H^1(\O)^d\,,\;W_{\disc}(\bpsi)\lesssim \norm{\bpsi}{H^1(\O)^d}h_\mesh.
\end{align}
Using \eqref{est.grad}, we readily check that
\be\label{comp.D.Ds}
\forall v\in X_{\disc,0}\,,\;\norm{\Pi_\disc v-\Pi_{\discs}v}{L^2(\O)}\le
h_\mesh\norm{\nabla_\disc v}{L^2(\O)^d}.
\ee
Therefore, as a consequence of \eqref{est.CD}--\eqref{est.WD} and of \cite[Remark 7.49]{koala},
\begin{align}
\label{est.CDs}
&C_\discs\lesssim 1\,,\\
\label{est.SDs}
&\forall \phi\in H^2(\O)\cap H^1_0(\O)\,,\; S_\discs(\phi)\lesssim\norm{\phi}{H^2(\O)}h_\mesh\,,\\
\label{est.WDs}
&\forall\bpsi\in H^1(\O)^d\,,\;W_{\discs}(\bpsi)\lesssim \norm{\bpsi}{H^1(\O)^d}h_\mesh.
\end{align}

We can now turn to the proof of the super-convergence result for the
modified HMM gradient scheme based on $\discs$.

\begin{proof}[Proof of Theorem \ref{th.supercv.discs}]

Properties \eqref{est.SDs}--\eqref{est.WDs} show that, for all $\phi\in H^2(\O)\cap
H^1_0(\O)$ such that $A\nabla\phi\in H^1(\O)^d$,
\be\label{est.WS.discs}
\WS_\discs(\phi)\lesssim \left(\norm{A\nabla\phi}{H^1(\O)^d}+
\norm{\phi}{H^2(\O)}\right)h_\mesh.
\ee
Hence, estimate \eqref{cvgrad-discs} is a consequence of \eqref{est.error.gs} in Theorem \ref{th:error.est.PDE}, of
\eqref{est.CDs} and of \eqref{assump.H2}. To prove \eqref{sc-discs}, we use 
the improved $L^2$ estimate for gradient schemes (Theorem \ref{th-gs-super}). Let us assume that we find a
mapping $\P_\discs: H^2(\O)\cap H^1_0(\O)\to X_{\disc,0}$ such that,
for any $\phi\in H^2(\O)\cap H^1_0(\O)$,
\be\label{interp.h2}
\hI_{\discs,h_\mesh}(\phi,\P_\discs\phi)\lesssim \norm{\phi}{H^2(\O)} h_\mesh^2
\ee
and
\be\label{tW.h2}
\forall \bpsi\in H^1(\O)^d\,,\;
\left|\tW_\discs(\bpsi,\P_\discs\phi)\right|
\lesssim \norm{\phi}{H^2(\O)}\norm{\bpsi}{H^1(\O)^d}h_\mesh^2.
\ee
Then, the proof of \eqref{sc-discs} is concluded by
applying Theorem \ref{th-gs-super} with this interpolant $\P_\discs$ (and $\alpha=h_\mesh$),
by using \eqref{assump.H2} on $\bu$ and $\varphi_g$ (recall that $\norm{g}{L^2(\O)}=1$),
and by invoking \eqref{est.WS.discs}. We now turn to the construction of $\P_\discs$
and to the proof of its properties.

If $\phi\in H^2(\O)$, then $\phi$ is continuous (since $d\le 3$) and
we can therefore set $\P_\discs\phi=((\phi_\cell)_{\cell\in\mesh},
(\phi_\edge)_{\edge\in\edgescv})\in X_{\disc,0}$ with
\[
\forall \cell\in\mesh\,,\;\phi_\cell=\phi(\x_\cell)\,;\quad
\forall \edge\in\edges\,,\;\phi_\edge=\frac{1}{|\edge|}\int_\edge \phi(\x)\d s(\x).
\]

\medskip

\textbf{Step 1}: {\it Proof of \eqref{interp.h2}.}

Let $\cell\in\mesh$. By \cite[Lemma B.1]{koala}, $\cell$ is star-shaped with respect to
all points in the ball $B(\x_\cell,\min_{\edge\in\edgescv}d_{\cell,\edge})
\supset B(\x_\cell,\theta_\polyd^{-1}h_\cell)$. Hence we can apply Lemma \ref{lem:app.H2.lin} (see Appendix)  with $V=\cell$. Let $L_\phi$ be the affine map given by this lemma.
Let $\gr_\cell=(\gr_\cell^\cell,
(\gr_\cell^\edge)_{\edge\in\edgescv})$ be a $\mathbb{P}_1$-exact gradient reconstruction on $\cell$ upon $S=(\x_\cell,(\centeredge)_{\edge\in\edgescv})$ (see Definition \ref{def:linex} with $U=\cell$).
We let 
\be\label{def:ell}
\ell=(L_\phi(\x_\cell),(L_\phi(\centeredge))_{\edge\in\edgescv}).
\ee
With abuse of notations, write $\gr_\cell \P_\discs\phi$ for $\gr_\cell(\phi_\cell,(\phi_\edge)_{\edge\in\edgescv})$.
Since $\gr_\cell$ is a $\mathbb{P}_1$-exact gradient reconstruction, we have $\gr_\cell \ell=\nabla L_\phi$ and thus
$\gr_\cell \P_\discs\phi -\nabla\phi=\gr_\cell \P_\discs\phi-\gr_\cell \ell
+\nabla L_\phi-\nabla\phi$.
By property of the norm $\norm{\gr_\cell}{}$ in Definition \ref{def:linex},
\begin{align}
&\norm{\gr_\cell \P_\discs\phi -\nabla\phi}{L^2(\cell)^d}\nonumber\\
&\qquad\le
h_\cell^{-1}|\cell|^{\frac{1}{2}}\norm{\gr_\cell}{}\max\left(|L_\phi(\x_\cell)-\phi_\cell|,
\max_{\edge\in\edgescv}|L_\phi(\centeredge)
-\phi_\edge|\right)\nonumber\\
&\qquad\quad+\norm{\nabla L_\phi-\nabla\phi}{L^2(\cell)^d}.
\label{est.gr.general}
\end{align}
Since $L_\phi$ is affine and $\centeredge$ is the center of mass of $\edge$,
by definition of $\phi_\edge$ and using estimate
\eqref{lem:app.H2.lin.est1}, we have
\[
|L_\phi(\centeredge) -\phi_\edge|=\left|\frac{1}{|\edge|}\int_\edge (L_\phi(\x)-\phi(\x))\d s(\x)\right|
\lesssim h_\cell^2 |\cell|^{-\frac{1}{2}}\norm{\phi}{H^2(\cell)}.
\]
We also have $|L_\phi(\x_\cell)-\phi_\cell|\lesssim h_\cell^2 |\cell|^{-\frac{1}{2}}\norm{\phi}{H^2(\cell)}$.
Plugging into \eqref{est.gr.general} and using \eqref{lem:app.H2.lin.est2}, this gives
\be\label{est.gr.lle}
\norm{\gr_\cell \P_\discs\phi -\nabla\phi}{L^2(\cell)^d}
\lesssim (1+\norm{\gr_\cell}{})h_\cell \norm{\phi}{H^2(\cell)}.
\ee
The restriction $(\nabla_\disc)_{|\cell}$ of $\nabla_\disc$ to the degrees of freedom 
$(v_\cell,(v_\edge)_{\edge\in\edgescv})$ in the cell $\cell$ 
is a $\mathbb{P}_1$-exact
gradient reconstruction on $\cell$ upon $S$ \cite[Section 3.6]{DEH15} and satisfies
$\norm{(\nabla_\disc)_{|\cell}}{}\lesssim 1$ (see \cite[Lemma 12.8]{koala}).
Hence, applying \eqref{est.gr.lle} to $\gr_\cell=(\nabla_\disc)_{|\cell}$,
squaring the resulting inequality, and summing over $\cell\in\mesh$, we find
\be\label{est.interp.grad}
\norm{\nabla_\disc\P_\discs\phi-\nabla\phi}{L^2(\O)^d}
\lesssim h_\mesh \norm{\phi}{H^2(\O)}.
\ee

Let us now estimate $\Pi_\discs\P_\discs\phi -\phi$. We still take $\ell$ as defined
by \eqref{def:ell}. Since $\nablacell$ is a $\mathbb{P}_1$-exact gradient reconstruction
upon $(\x_\cell,(\centeredge)_{\edge\in\edgescv})$ (see \cite[Lemma B.6]{koala}), we have
$\nablacell\ell=\nabla L_\phi$ and thus, for any $\x\in \cell$,
by the definition \eqref{def.Pidiscs} of $\Pi_\discs$,
\[
\Pi_\discs \ell(\x)=L_\phi(\x_\cell)+\nabla L_\phi\cdot(\x-\x_\cell)=L_\phi(\x).
\]
Thus, using the estimate \eqref{lem:app.H2.lin.est1}, we have, for all $\x\in \cell$,
\begin{align*}
|\Pi_\discs &\P_\discs\phi(\x)-\phi(\x)|\\
&\le
|\Pi_\discs\P_\discs\phi(\x) -\Pi_\discs \ell(\x)|
+|L_\phi(\x)-\phi(\x)|\\
&\le |\phi(\x_\cell)-L_\phi(\x_\cell)|+h_\cell|\nablacell\P_\discs\phi-\nablacell \ell|
+\sup_{\overline{\cell}}|\phi-L_\phi|\\
&\lesssim  h_\cell^2|\cell|^{-\frac{1}{2}}\norm{\phi}{H^2(\cell)}+h_\cell|\nablacell\P_\discs\phi-\nabla\phi|
+h_\cell|\nabla\phi-\nabla L_\phi|.
\end{align*}
Taking the $L^2(\cell)$ norm leads to
\begin{align}
\norm{\Pi_\discs \P_\discs\phi-\phi}{L^2(\cell)}
\lesssim{}&  h_\cell^2\norm{\phi}{H^2(\cell)}
+h_\cell\norm{\nablacell\P_\discs\phi-\nabla\phi}{L^2(\cell)^d}\nonumber\\
&+h_\cell\norm{\nabla\phi-\nabla L_\phi}{L^2(\cell)^d}.
\label{est.Id.last}
\end{align}
It is easy to see that the $\mathbb{P}_1$-exact gradient reconstruction $\gr_\cell=\nablacell$
satisfies $\norm{\gr_\cell}{}\lesssim 1$ (see for example \cite[Lemma B.6]{koala},
or use \eqref{est.grad} and $\norm{(\nabla_\disc)_{|\cell}}{}\lesssim 1$
as mentioned previously).
Hence, applying \eqref{est.gr.lle} to $\gr_\cell=\nablacell$ and using \eqref{lem:app.H2.lin.est2},
we obtain
$\norm{\Pi_\discs \P_\discs\phi-\phi}{L^2(\cell)}\lesssim h_\cell^2\norm{\phi}{H^2(\cell)}$.
Squaring and summing over $\cell\in\mesh$ yields
$\norm{\Pi_\discs \P_\discs\phi-\phi}{L^2(\O)}\lesssim h_\mesh^2\norm{\phi}{H^2(\O)}$
which, combined with \eqref{est.interp.grad}, concludes the proof of \eqref{interp.h2}.

\medskip

\textbf{Step 2}: {\it Proof of \eqref{tW.h2}.}

Since $\phi\in H^1_0(\O)$, by Stokes formula we can write
\begin{align}
\tW_\discs(\bpsi,\P_\discs\phi)&=
\int_\O \left( \Pi_\discs \P_\discs\phi \: \div \bpsi
+\nabla_\disc \P_\discs\phi\cdot\bpsi \right) \d\x\nonumber\\
&=\int_\O \left( (\Pi_\discs \P_\discs\phi -\phi)\div \bpsi
+(\nabla_\disc \P_\discs\phi-\nabla\phi)\cdot\bpsi \right) \d\x.
\label{tW.introduce.phi}
\end{align}
Using \eqref{interp.h2}, this gives
\begin{align}
\left|\tW_\discs(\bpsi,\P_\discs\phi)\right|
\lesssim{}& h_\mesh^2\norm{\phi}{H^2(\O)}\norm{\div\bpsi}{L^2(\O)}\nonumber\\
&+\sum_{\cell\in\mesh}\left|\int_\cell (\nabla_\disc \P_\discs\phi-\nabla\phi)\cdot\bpsi \d\x\right|.
\label{est.tW.1}
\end{align}
We now work on the last term in this estimate. By \eqref{average.grad} and choice
of $\phi_\edge$ we have
\begin{align}
\int_\cell \nabla_\disc \P_\discs\phi\d\x=|\cell|\nablacell\P_\discs\phi
&=\sum_{\edge\in\edgescv}|\edge|\phi_\edge\bfn_{\cell,\edge}\nonumber\\
&=\sum_{\edge\in\edgescv}\int_\edge \phi(\x)\bfn_{\cell,\edge}\d\x
=\int_\cell \nabla\phi\d\x.
\label{trick.average}
\end{align}
Hence, if $\displaystyle \bpsi_\cell=\frac{1}{|\cell|}\int_\cell \bpsi(\x)\d\x$,
\begin{align*}
\left|\int_\cell (\nabla_\disc \P_\discs\phi-\nabla\phi)\cdot\bpsi\d\x\right|
&=\left|\int_\cell (\nabla_\disc \P_\discs\phi-\nabla\phi)\cdot(\bpsi-\bpsi_\cell)\d\x\right|\\
&\le \norm{\nabla_\disc\P_\discs\phi-\nabla\phi}{L^2(\cell)^d}
\norm{\bpsi-\bpsi_\cell}{L^2(\cell)^d}.
\end{align*}
It is quite classical (see, e.g., \cite[Lemma B.7]{koala}) that, for all $\psi\in H^1(\cell)$,
if $\psi_K=\frac{1}{|K|}\int_K \psi(\x)\d\x$ then
\be\label{est.psi}
\norm{\psi-\psi_\cell}{L^2(\cell)}\lesssim h_\cell\norm{\psi}{H^1(\cell)}.
\ee
Applying this to each component of $\bpsi$ and using \eqref{est.gr.lle} with $\gr_\cell=
(\nabla_\disc)_{|\cell}$ yields
\[
\left|\int_\cell (\nabla_\disc \P_\discs\phi-\nabla\phi)\cdot\bpsi\d\x\right|
\lesssim h_\cell^2\norm{\phi}{H^2(\cell)}\norm{\bpsi}{H^1(\cell)^d}.
\]
Summing over $\cell$, using the Cauchy--Schwarz inequality, and plugging the result
in \eqref{est.tW.1}, we obtain \eqref{tW.h2}.
\end{proof}

\section{Super-convergence result for HMM and TPFA schemes}\label{sec:imp.HMM}

If $(\x_\cell)_{\cell\in\mesh}$ are the centers of mass of the cells, HMM
methods are hMFD methods and the super-convergence is therefore known.
In some instances, however, it is interesting to choose $\x_\cell$ not at the center of mass
of $\cell$. On triangles, for example, choosing $\x_\cell$ as the circumcenter of
$\cell$ allows to recover the two-point flux approximation finite volume scheme. This ensures that the scheme satisfies a discrete maximum principle, has a very small matrix stencil, etc.
Our aim is to show that, if $(\x_\cell)_{\cell\in \mesh}$ are not the centers of mass,
but if we can create patches of cells over which, \emph{on average}, these points are (close to)
the centers of mass, then a super-convergence result still occurs for HMM methods.
We first define this notion of patches of cells. Recall that a set $X$ is
star-shaped with respect to a subset $Y$ if, for all $\x\in X$ and all $\y\in Y$, the segment
$[\x,\y]$ is contained in $X$.

\begin{definition}[Patching of cells]\label{def.patching}
Let $\polyd$ be a polytopal mesh of $\O$ in the sense of Definition \ref{def:polymesh}.
A patching of the cells of $\polyd$ is a family $\groupcell$ of disjoint sets of cells (the patches), such that,
for each patch $\patch\in\groupcell$, letting $U_\patch:=\cup_{\cell\in \patch}\cell$, there exists a ball $B_\patch\subset U_\patch$ such that $U_\patch$ is star-shaped with respect to $B_\patch$.

We then define:
\begin{itemize}
\item $\Omega_\groupcell=\cup_{\patch\in\groupcell}\cup_{K\in\patch}K$ the region
of $\Omega$ covered by the patches, and $\Omega_\groupcell^c=\Omega\backslash\Omega_\groupcell$ its complement,
\item the regularity factor by
\be\label{def:muG}
\mu_\groupcell = \max_{\patch\in\groupcell}{\rm Card}(\patch)+
\max_{\patch\in\groupcell} \max_{K\in \patch}\frac{h_K}{{\rm diam}(B_\patch)}
\ee
where ${\rm Card}(\patch)$ is the number of cells in $\patch$,
\item the maximum norm of the patch-averaged vector between
the centers of masses and the cell points by
\[
e_\groupcell=\max_{\patch\in \groupcell}\left|\frac{1}{|U_\patch|}\sum_{\cell\in \patch}|\cell|(\overline{\x}_\cell-\x_\cell)\right|.
\]
\end{itemize}
\end{definition}

\begin{remark} For any $\patch\in\groupcell$, since $U_\patch$ is connected,
the diameter of $U_\patch$ (and thus of $B_\patch$) is bounded above by $\mu_\groupcell h_K$ for all $K\in \patch$.
Hence, bounding above $\mu_\groupcell$ requires in particular that, for any $\patch\in\groupcell$,
the diameter of $B_\patch$ is comparable to the diameter of any $K\in \patch$.
\label{rem:reg.groupcell}
\end{remark}

Using this notion, we state our super-convergence result for HMM schemes.
The result varies from some ``classical'' ones as it does not directly
involve the $L^2$ projection of the exact function on the piecewise constant functions,
but pointwise values of the exact function. This is in a sense expected, as the $L^2$
projection is an appropriate operator only when $(\x_\cell)_{\cell\in\mesh}$ are
the centers of mass (only case when the $L^2$ projection of $\bu$ is $h_\cell^2$ close
to the pointwise values $\bu(\x_\cell)$).

In the following theorem and its proof, we use the notation:
\be\label{notation.lesssim.hMFD.group}
\begin{aligned}
&\mbox{$\mathcal A \lesssim \mathcal B$ means that 
$\mathcal A\le C\mathcal B$ with $C$ depending only on}\\
&\mbox{$\O$, $A$ and an upper bound
of $\theta_\polyd+\zeta_\disc+\mu_\groupcell$}
\end{aligned}
\ee
and we call a ``strip of witdh $\rho>0$'' a set of the form $\strip_{H}(\rho):=\{x\in\O\,:\,
\dist(x,H)\le \rho\}$, where $H$ is an hyperplane of $\R^d$.

\begin{theorem}[Super-convergence for HMM schemes]\label{th-supercv-HMM-group}~

Under the assumptions \eqref{assump}, \eqref{assump.H2}, and that $d\le 3$,
let $f\in H^1(\O)$ and $\bu$ be the solution to
\eqref{weak_state}. Choose $\polyd$ to be  a polytopal mesh in the sense of
Definition \ref{def:polymesh}. Let $\groupcell$ be a patching of the cells of $\polyd$
such that $\Omega_\groupcell^c$ is contained in the union of $r_\groupcell$ strips
of width $\rho_\groupcell$.

Take $\disc$ an HMM gradient discretisation on $\polyd$
(see Definition \ref{def.HMM.GD}) and let $u_{\disc}$ be the solution of the corresponding
gradient scheme \eqref{base.GS}. Let $\bu_\centers$ be the piecewise constant
function on $\mesh$ equal to $\bu(\x_K)$ on $K\in\mesh$ (see \eqref{def.ucenters}).
Then,
\be\label{sc-disc-group}
\norm{\Pi_{\disc} u_{\disc}-\bu_\centers}{L^2(\O)}\lesssim
\norm{f}{H^1(\O)}\left(h_\mesh^2+e_\groupcell+r_\groupcell \rho_\groupcell h_\mesh\right).
\ee
In particular, if $\groupcell$ is such that $e_\groupcell\lesssim h_\mesh^2$,
$r_\groupcell\lesssim 1$ and $\rho_\groupcell\lesssim h_\mesh$, then the
following super-convergence estimate occurs:
\[
\norm{\Pi_{\disc} u_{\disc}-\bu_\centers}{L^2(\O)}\lesssim
\norm{f}{H^1(\O)} h_\mesh^2.
\]
\end{theorem}

\begin{remark}[Super-convergence for $\mathbb{RT}_0$]
The $\mathbb{RT}_0$--$\mathbb{P}_0$ mixed finite element me\-thod is
a particular instance of an HMM scheme with $(\x_K)_{K\in\mesh}$ being the centers of mass
of the cells \cite{CAN08}. Hence, Theorem \ref{th-supercv-HMM-group} with 
the trivial patching $\groupcell=\mesh$ yields
a super-convergence result for the $\mathbb{RT}_0$--$\mathbb{P}_0$ scheme.
This result was previously established in \cite{DM85}.
\end{remark}

\begin{proof}

The proof hinges on two tricks. In Step 1, letting $u_\discs$ be the solution to the
modified HMM scheme, we show that $ \norm{\Pi_\disc u_\disc-\Pi_\disc u_\discs}{L^2(\O)}$
is of order $\mathcal O(h_\mesh^2+e_\groupcell+r_\groupcell \rho_\groupcell h_\mesh)$. In Step 2, we introduce a
weighted projection $\pi^w_\mesh$ such that $\Pi_\disc u_\discs=\pi^w_\mesh(\Pi_\discs u_\discs)$
and $\norm{\pi^w_\mesh \bu-\bu_\centers}{L^2(\O)}=\mathcal O(h_\mesh^2)$.
Combined with the result from Step 1 and the super-convergence property \eqref{sc-discs} of the modified
HMM method, this concludes the proof.

\medskip

\textbf{Step 1}: \emph{Comparison of the solutions to the schemes for $\disc$ and $\discs$}.

Let $u_\discs$ be the solution to \eqref{base.GS} with $\discs$ instead of $\disc$.
Subtracting the two gradient schemes corresponding to $\discs$ and $\disc$ we see that,
for all $v_\disc\in X_{\disc,0}$,
\begin{multline}\label{GS.sub}
\int_\O A(\x)(\nabla_\disc u_\discs-\nabla_\disc u_\disc)(\x)\cdot \nabla_\disc v_\disc(\x)\d\x
\\
=\int_\O f(\x)(\Pi_\discs v_\disc-\Pi_\disc v_\disc)(\x)\d\x.
\end{multline}
Let $\e_\cell=\overline{\x}_\cell-\x_\cell$.
Using the definition \eqref{def.Pidiscs} of $\Pi_\discs$, we can write
\begin{align}
&\int_\O f(\x)(\Pi_\discs v_\disc-\Pi_\disc v_\disc)(\x)\d\x\nonumber\\
&=\sum_{\cell\in\mesh}\left(\int_\cell f(\x) (\x-\x_\cell)\d\x\right)\cdot\nablacell v_\disc\nonumber\\
&=\sum_{\cell\in\mesh}\left(\int_\cell f(\x) (\x-\overline{\x}_\cell)\d\x\right)\cdot\nablacell v_\disc
+\sum_{\cell\in\mesh}\left(\int_\cell f(\x) \e_\cell\d\x\right)\cdot\nablacell v_\disc.
\label{for.grouping}
\end{align}
Let $g\in L^2(\O)$ and $v_\disc$ be the solution to the gradient
scheme \eqref{base.GS} with right-hand side $g$ instead of $f$. Combining
\eqref{GS.sub}, \eqref{for.grouping} and the definition of $v_\disc$, we find
\begin{align}\label{group.1}
& \int_\O g(\x)(\Pi_\disc u_\discs-\Pi_\disc u_\disc)(\x)\d\x
=\int_\O f(\x)(\Pi_\discs v_\disc-\Pi_\disc v_\disc)(\x)\d\x  \nonumber\\
& \qquad =\sum_{\cell\in\mesh}\left(\int_\cell f(\x) (\x-\overline{\x}_\cell)\d\x\right)\cdot\nablacell v_\disc 
  +\sum_{\cell\in\mesh}\left(\int_\cell f(\x) \e_\cell\d\x\right)\cdot\nablacell v_\disc\nonumber \\
  & \qquad =:T_1+T_2.
\end{align}
Since $\overline{\x}_\cell$ is the center of mass of $\cell$, we have $\int_\cell (\x-\overline{\x}_\cell)\d\x=0$. Hence,
letting $f_\cell$ to be the average value of $f$ on $\cell$, using \eqref{est.psi} and
\eqref{est.grad}, $T_1$ is estimated:
\begin{align}
T_1={}&\sum_{\cell\in\mesh}\left(\int_\cell (f(\x)-f_\cell) (\x-\overline{\x}_\cell)\d\x\right)\cdot\nablacell v_\disc \nonumber\\
\le{}&\sum_{\cell\in\mesh}h_\cell\norm{f-f_\cell}{L^2(\cell)} |\cell|^{\frac{1}{2}}|\nablacell v_\disc|\nonumber\\
\lesssim{}&h_\mesh^2\sum_{\cell\in\mesh}\norm{f}{H^1(\cell)} \norm{\nabla_\disc v_\disc}{L^2(\cell)^d}\nonumber\\
\lesssim{}& h_\mesh^2\norm{f}{H^1(\O)} \norm{\nabla_\disc v_\disc}{L^2(\O)^d}\lesssim
h_\mesh^2 \norm{f}{H^1(\O)}\norm{g}{L^2(\O)}.
\label{group.T1}
\end{align}
In the last line, we used the discrete Cauchy--Schwarz inequality and the stability property \eqref{est.stability} in Theorem \ref{th:error.est.PDE}.

We now estimate $T_2$. Let $\bv$ be the solution to \eqref{weak_state} with right-hand side $g$ instead of $f$,
and recall that $v_\disc$ is the solution to the gradient scheme for this continuous problem.
Hence, $\nabla_\cell v$ should be close to $\nabla\bv$ on the cell $\cell$. 
To use this proximity, write
\begin{align}
T_2={}&\sum_{\cell\in\mesh}\int_\cell f(\x) \e_\cell \cdot(\nablacell v_\disc-\nabla\bv(\x))\d\x
+\sum_{\cell\in\mesh_\groupcell}\int_\cell (f\nabla\bv)(\x)\cdot \e_\cell \d\x\nonumber\\
{}&+\sum_{\cell\not\in\mesh_\groupcell}\int_\cell (f\nabla\bv)(\x)\cdot \e_\cell \d\x
\nonumber\\
={}&T_{2,1}+T_{2,2}+T_{2,3},
\label{group.T2}
\end{align}
where $\mesh_\groupcell=\cup_{\patch\in\groupcell}\patch$ is the set of cells covered
by the patches.
Letting $\nabla_\mesh v$ be the piecewise constant function on $\mesh$ equal to $\nablacell v$
on each $\cell\in\mesh$, the bound $\sup_{\cell\in\mesh}|\e_\cell|\le h_\mesh$ gives
\[
|T_{2,1}|\le h_\mesh\norm{f}{L^2(\O)}\norm{\nabla_\mesh v_\disc-\nabla\bv}{L^2(\O)^d}.
\]
By \eqref{average.grad}, $\nabla_\mesh v_\disc=\pi_\mesh(\nabla_\disc v_\disc)$, where
$\pi_\mesh$ is the $L^2$ projector on piecewise constant functions (here, it is used
component-wise). Since $\pi_\mesh$
has norm $1$, the estimates \eqref{est.error.gs}, \eqref{est.CD}--\eqref{est.WD},
\eqref{est.psi} and the $H^2$ regularity property \eqref{assump.H2} for $\bv$ yield
\begin{align}
|T_{2,1}|\lesssim{}& h_\mesh \norm{f}{L^2(\O)}\left(
\norm{\pi_\mesh(\nabla_\disc v_\disc -\nabla\bv)}{L^2(\O)^d}
+\norm{\pi_\mesh(\nabla\bv) -\nabla\bv}{L^2(\O)^d}\right)\nonumber\\
\lesssim{}& h_\mesh \norm{f}{L^2(\O)}\left(
\norm{\nabla_\disc v_\disc -\nabla\bv}{L^2(\O)^d}
+h_\mesh\norm{\bv}{H^2(\O)}\right)\nonumber\\
\lesssim{}& h_\mesh^2 \norm{f}{L^2(\O)}\norm{g}{L^2(\O)}.
\label{group.T21}
\end{align}
The sum $T_{2,2}$ is estimated by using the patches. Let $\patch\in\groupcell$,
$K\in \patch$, and apply Lemma \ref{lem.est.moy} (see Appendix)
twice to $(U,V,O)=(K,B_\patch,U_\patch)$ and $(U,V,O)=
(U_\patch,B_\patch,U_\patch)$. Owing to Remark \ref{rem:reg.groupcell} and using the
upper bound on $\theta_\polyd$,
\begin{align*}
&\diam(U_\patch)\lesssim h_K\,,\quad\diam(U_\patch)^d\lesssim \diam(B_\patch)^d\lesssim |B_\patch|\le
|U_\patch|\\
&\mbox{and } |K|\lesssim h_K^d\lesssim \diam(B_\patch)^d\lesssim |B_\patch|.
\end{align*}
Hence,
\begin{align*}
\Big|\frac{1}{|K|}\int_K (f\nabla\bv)(\x)\d\x
&-\frac{1}{|U_\patch|}\int_{U_\patch}(f\nabla\bv)(\x)\d\x\Big|\\
\le{}&\Big|\frac{1}{|K|}\int_K (f\nabla\bv)(\x)\d\x
-\frac{1}{|B_\patch|}\int_{B_\patch}(f\nabla\bv)(\x)\d\x\Big|\\
&+\Big|\frac{1}{|B_\patch|}\int_{B_\patch} (f\nabla\bv)(\x)\d\x
-\frac{1}{|U_\patch|}\int_{U_\patch}(f\nabla\bv)(\x)\d\x\Big|\\
\lesssim{}&\left(\frac{{\rm diam}(U_\patch)^{d+1}}{|K|\,|B_\patch|}
+\frac{{\rm diam}(U_\patch)^{d+1}}{|B_\patch|\,|U_\patch|}\right)
\norm{f\nabla\bv}{W^{1,1}(U_\patch)^d}\\
\lesssim{}& \frac{h_\mesh}{|K|}\norm{f\nabla\bv}{W^{1,1}(U_\patch)^d}.
\end{align*}
Since the patches are pairwise disjoint and $\left|\sum_{\cell\in \patch} |\cell|\e_\cell\right|\le
|U_\patch|e_\groupcell$, by the $H^2$ regularity property \eqref{assump.H2},
\begin{align}
T_{2,2}={}&\sum_{\patch\in\groupcell} \sum_{K\in \patch} |K|\e_\cell\cdot\left(\frac{1}{|K|}\int_K (f\nabla\bv)(\x)\d\x\right)\nonumber\\
\lesssim{}&\sum_{\patch\in\groupcell} \sum_{K\in \patch} h_\mesh^2\norm{f\nabla\bv}{W^{1,1}(U_\patch)^d}\nonumber\\
{}&+\sum_{\patch\in\groupcell} \left(\frac{1}{|U_\patch|}\int_{U_\patch} (f\nabla\bv)(\x)\d\x\right)
\cdot \sum_{\cell\in \patch} |\cell|\e_\cell\nonumber\\
\lesssim{}&h_\mesh^2\norm{f\nabla\bv}{W^{1,1}(\O)^d}+e_\groupcell\norm{f\nabla\bv}{L^1(\O)^d}\nonumber\\
\lesssim{}&(h_\mesh^2+e_\groupcell)\norm{f}{H^1(\O)}\norm{\bv}{H^2(\O)}
\lesssim (h_\mesh^2+e_\groupcell)\norm{f}{H^1(\O)}\norm{g}{L^2(\O)}.
\label{estim.T22}\end{align}
To estimate $T_{2,3}$, let $(\strip_{H_i}(\rho_\groupcell))_{1\le i\le r_\groupcell}$ be the the strips covering $\Omega_\groupcell^c$.
Since $|\e_\cell|\le h_\mesh$, Lemma \ref{lemma.strip} yields
\begin{multline*}
|T_{2,3}|\le h_\mesh\int_{\Omega_\groupcell^c} |f\nabla\bv|\d\x
\le h_\mesh\sum_{i=1}^{r_\groupcell} \norm{f\nabla \bv}{L^1(\strip_{H_i}(\rho_\groupcell))^d}\\
\lesssim r_\groupcell \rho_\groupcell h_\mesh \norm{f\nabla \bv}{W^{1,1}(\O)^d}\le r_\groupcell \rho_\groupcell h_\mesh\norm{f}{H^1(\O)}\norm{g}{L^2(\O)}.
\end{multline*}
Plugging this estimate, \eqref{group.T21} and \eqref{estim.T22} in \eqref{group.T2} gives
\[
T_2\lesssim (h_\mesh^2+e_\groupcell+r_\groupcell\rho_\groupcell h_\mesh)\norm{f}{H^1(\O)}\norm{g}{L^2(\O)}
\]
which, combined with \eqref{group.T1} and \eqref{group.1}, leads to
\[
\int_\O g(\x)(\Pi_\disc u_\discs-\Pi_\disc u_\disc)(\x)\d\x
\lesssim (h_\mesh^2+e_\groupcell+r_\groupcell\rho_\groupcell h_\mesh)\norm{f}{H^1(\O)}\norm{g}{L^2(\O)}.
\]
Take the supremum over $g\in L^2(\O)$ with norm 1 to find
\be\label{group.end.step1}
\norm{\Pi_\disc u_\discs-\Pi_\disc u_\disc}{L^2(\O)}
\lesssim (h_\mesh^2+e_\groupcell+r_\groupcell\rho_\groupcell h_\mesh)\norm{f}{H^1(\O)}.
\ee

\medskip

\textbf{Step 2}: \emph{a weighted projection, and conclusion}.

Let $(w_\cell)_{\cell\in\mesh}$ be the functions given by Lemma \ref{lem.w.xK}
(see Appendix).
We define $\Pmesh:L^2(\O)\to L^2(\O)$ by
\[
\forall \phi\in L^2(\O)\,,\;\forall \cell\in\mesh\,,\;
(\Pmesh \phi)_{|K}= \frac{1}{|\cell|}\int_\cell \phi(\x)w_\cell(\x)\d\x.
\]
We have $\norm{\Pmesh}{L^2(\O)\to L^2(\O)}\le \max_{\cell\in\mesh}\norm{w_\cell}{L^\infty(\cell)}
\lesssim 1$. Moreover, by definition of $w_\cell$ and of $\Pi_\discs$, we have, on any
$\cell\in\mesh$,
\[
\Pmesh(\Pi_\discs u_\discs)=\frac{1}{|\cell|}\int_\cell (\Pi_\disc u_\discs + \nablacell
u_\discs\cdot (\x-\x_K))w_\cell(\x)\d\x=\Pi_\disc u_\discs.
\]
Hence, by \eqref{sc-discs},
\begin{align}
\norm{\Pi_\disc u_\discs-\Pmesh \bu}{L^2(\O)}
={}&\norm{\Pmesh(\Pi_\discs u_\discs-\bu)}{L^2(\O)}\nonumber\\
\lesssim{}& \norm{\Pi_\discs u_\discs-\bu}{L^2(\O)}
\lesssim h_\mesh^2\norm{f}{L^2(\O)}.
\label{est.PM}\end{align}
Using \eqref{est.phi.w} in Lemma \ref{lem.w.xK}, we have $\norm{\bu_\centers-\Pmesh\bu}{L^2(\O)}
\lesssim h_\mesh^2\norm{\bu}{H^2(\O)}$. Combined with \eqref{est.PM} and
using the $H^2$ regularity property \eqref{assump.H2}, this gives
\be\label{pt.superconv}
\norm{\Pi_\disc u_\discs-\bu_\centers}{L^2(\O)} \lesssim h_\mesh^2\norm{f}{L^2(\O)}.
\ee
The conclusion follows from this estimate and from \eqref{group.end.step1}. \end{proof}

We can now prove the super-convergence of TPFA schemes on triangular meshes.

\begin{proof}[Proof of Theorem \ref{th:super_cv.tpfa}]

On a classical TPFA triangulation as in Definition \ref{def:classical.tpfa.tri},
the TPFA scheme is an HMM scheme $\disc$, with
$\zeta_\disc$ depending only on an upper bound of $\theta_\polyd$ \cite[Section 5.3]{dro-10-uni}.
Theorem \ref{th-supercv-HMM-group} therefore applies. Notice that $\theta_\polyd$
is bounded by a constant only depending on \itr{}.
To conclude the proof,
we describe, for each type of TPFA triangulation, a patching $\groupcell$
such that (i) $e_\groupcell=0$, (ii) $\mu_\groupcell\le C$ (see \eqref{def:muG}), where
$C$ depends only on an upper bound of $\theta_\polyd$,
and (iii) $\O_\groupcell^c$ is contained in $r$ strips of size $Mh_\mesh$,
where $r$ and $M$ only depend on \itr{}. Theorem \ref{th:super_cv.tpfa} then follows immediately
from the second conclusion in Theorem \ref{th-supercv-HMM-group}.
In the rest of this proof, we use the same notation $\e_\cell=\overline{\x}_K-\x_K$ as in the proof of
Theorem \ref{th-supercv-HMM-group}.

\medskip

\emph{Subdivision}. Most triangles of $\polyd$ can be patched in pairs forming rhombuses. Two of such pairs
are illustrated in grey in Figure \ref{fig:classical_subd}, left. In such a rhombus,
each triangle $K_i$, $i=1,2$, is the symmetric of the other with respect to the center of the rhombus.
As a consequence, $\e_{K_1}=-\e_{K_2}$ and the corresponding patching satisfies $e_\groupcell=0$.
The region $\O_\groupcell^c$ consists of one layer of triangles around each
edge of the initial triangulation \itr{} -- examples of such triangles
are the dotted triangles in Figure \ref{fig:classical_subd} (left) -- and is
therefore contained in a fixed number of strips of width $h_\mesh$.

\medskip

\emph{Reproduction by symmetry}. The cells are
patched in four contiguous reproductions of \itr, as shown in grey in Figure \ref{fig:classical_subd} (center). The symmetries ensure that, with such a patching, $e_\groupcell=0$.
For an odd number of symmetries, $\O_\groupcell^c=\emptyset$. For
an even number of symmetries, $\O_\groupcell^c$ is made of the reproductions of
\itr{} along two edges of $\O$, and is thus contained in two strips of size $Mh_\mesh$
with $M$ depending only on \itr{}.

\medskip

\emph{Reproduction by translation}. Obtaining a conforming triangulation of $\O$ with
a reproduction by translation of \itr{} imposes some symmetry properties on
this initial triangulation. The vertices on the left side of \itr{} must match the vertices on the right side of \itr, and similarly for the vertices on the top and bottom sides. 
Applying Lemma \ref{lem:compensation.triangles} in the appendix to the
unit square $Q$ shows that,
for such a triangulation, $\sum_{K\in \polyd_0}|K|\e_\cell=0$.
Indeed, each left boundary edge of \itr{} is matched by a right boundary edge
(same for top/bottom), and for these edges the quantities $|\vertex^\edge_i-\overline{\x}_Q|$ are identical whilst $\bfn_{Q,\edge}$ are opposite.
Hence, the patching $\groupcell$ made of the reproductions of the initial triangulation, as shown
in grey in Figure \ref{fig:classical_subd} (right), satisfies $e_\groupcell=0$
and $\O_\groupcell^c=\emptyset$.
\end{proof}

\begin{remark}
The property $\sum_{K\in \polyd_0}|K|\e_\cell=0$,
used in the proof above for meshes obtained by reproduction by translation, occurs with initial triangulations
of other polygons than the unit square. For example, if a conforming tessellation of a region is created by translating an elementary triangulation \itr{} of an hexagon $Q$, the edges of \itr{} on the opposite boundaries of this hexagon must
match and $\sum_{K\in \polyd_0}|K|\e_\cell=0$.
\end{remark}

\section{Numerical tests}\label{sec:num.tests}

In all the following tests, we consider \eqref{base} with $\O=(0,1)^2$,
$A={\rm Id}$, $\bu(x,y)=16x(1-x)y(1-y)$ and $f=-\Delta \bu$. We measure the following relative
$L^2$ errors of the HMM or TPFA schemes on $\bu$ and its gradient:
\[
\err_\disc(\bu)=\frac{\norm{\Pi_\disc u_\disc -\bu_\centers}{L^2(\O)}}{\norm{\bu_\centers}{L^2(\O)}}
\quad\mbox{ and }\quad
\err_\disc(\nabla\bu)=\frac{\norm{\nabla_\disc u_\disc -(\nabla\bu)_\centers}{L^2(\O)^d}}{\norm{(\nabla\bu)_\centers}{L^2(\O)^d}},
\]
where $(\nabla\bu)_\centers$ is the piecewise constant function equal, for all $K\in\mesh$,
to $\nabla\bu(\x_K)$ on $K$. These errors are plotted against the mesh size $h_\mesh$.
To test the super-convergence of 
the modified HMM method of Section \ref{sec:mod.hmm}, we use the following measure which,
according to \eqref{pt.superconv} (a direct consequence of the super-convergence
result \eqref{sc-discs} of $u_\discs$), should decrease as $h_\mesh^2$, even in
the absence of local compensation:
\[
\err_{\discs}(\bu)=\frac{\norm{\Pi_\disc u_\discs -\bu_\centers}{L^2(\O)}}{\norm{\bu_\centers}{L^2(\O)}}.
\]

\subsection{HMM method}\label{sec:tests.HMM}

The super-convergence for HMM schemes with $\centers$ given by the centers of mass of the cells has
already been numerically illustrated in a number of test cases, see e.g. \cite{HH08}. We rather
focus here on two cases where the points in $\centers$ are shifted away from the centers of mass.

\medskip

\textbf{Test 1}: \emph{Local compensation}

We consider a cartesian grid in which, every other cell, $\x_K$ is shifted to the top-right
or bottom-left of the centers of mass; see Figure \ref{fig:test1}, left. Grouping the cells by neighbourhing pairs, as represented
by the greyed area in this figure, gives a patching $\groupcell$ such that $e_\groupcell=0$ and $\O_\groupcell^c=\emptyset$.
Theorem \ref{th-supercv-HMM-group} therefore predicts the $\mathcal O(h_\mesh^2)$ estimate
on $\err_\disc(\bu)$ that is observed in Figure \ref{fig:test1}, right. The modified HMM method is also super-convergent and, quite naturally, beats the error of the HMM method (by a factor 2).

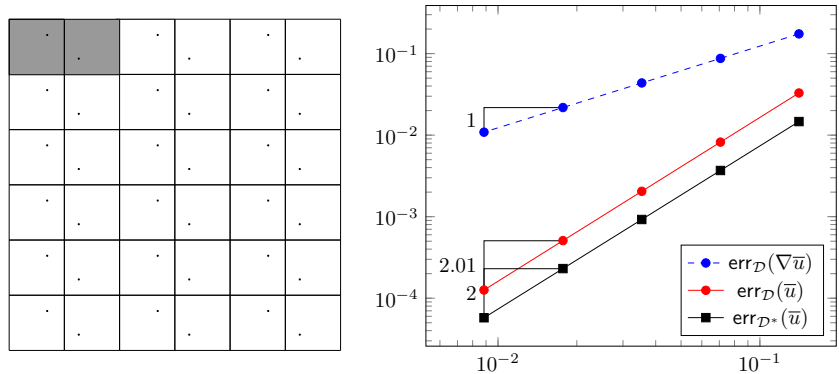
\begin{figure}[htb]
\begin{center}
\begin{tabular}{lr}
\raisebox{1em}{\resizebox{0.35\linewidth}{!}{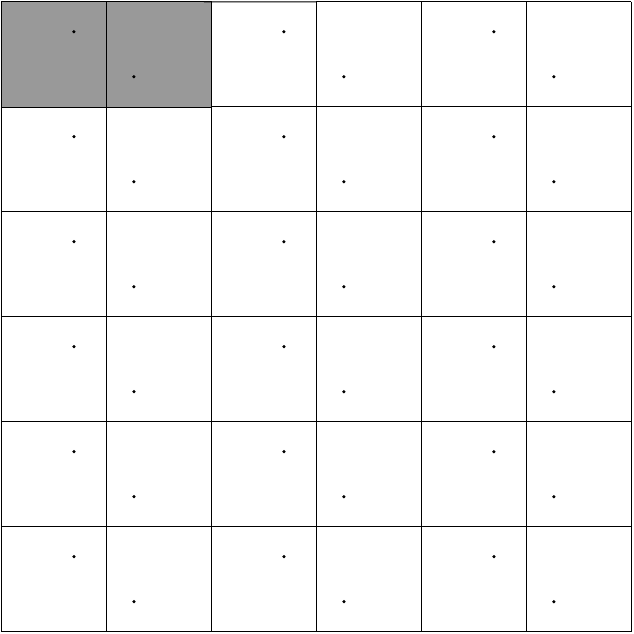}}
&
\resizebox{0.5\linewidth}{!}{%
	 \begin{tikzpicture}[scale=1]
			 \begin{loglogaxis}[
					 xmin = .00530400000000000000, 
					 legend style = {
					   legend pos = south east
					 }
				 ]
		     \addplot[dashed,mark=*,blue,mark options={solid}] table[x=meshsize,y={create col/linear regression={y=erL2_gradp}}] {HMM.test1.compensation.dat}
		     coordinate [pos=0.75] (A)
		     coordinate [pos=1.00] (B);
		     \xdef\slopeb{\pgfplotstableregressiona}
		     \draw (A) -| (B) node[pos=0.75,anchor=east] {\pgfmathprintnumber{\slopeb}};
		     \addplot[mark=*,red,mark options={solid}] table[x=meshsize,y={create col/linear regression={y=erL2_p}}] {HMM.test1.compensation.dat}
		     coordinate [pos=0.75] (A)
		     coordinate [pos=1.00] (B);
		     \xdef\slopea{\pgfplotstableregressiona}
		     \draw (A) -| (B) node[pos=0.75,anchor=east] {\pgfmathprintnumber{\slopea}};
		     \addplot[mark=square*,black,mark options={solid}] table[x=meshsize,y={create col/linear regression={y=erL2_p}}] {modifiedHMM.test1.supercv.dat}
		     coordinate [pos=0.75] (A)
		     coordinate [pos=1.00] (B);
		     \xdef\slopec{\pgfplotstableregressiona}
		     \draw (A) -| (B) node[pos=0.75,anchor=east] {\pgfmathprintnumber{\slopec}};
		     \legend{$\err_\disc(\nabla\bu)$,$\err_\disc(\bu)$,$\err_\discs(\bu)$};
		   \end{loglogaxis}
		 \end{tikzpicture}
}
\end{tabular}
\caption{Test 1: position of the points $\x_K$ (left), and rates of convergence (right) for the HMM and modified HMM methods.}
\label{fig:test1}
\end{center}
\end{figure}

\smallskip
\textbf{Test 2}: \emph{Loss of super-convergence for HMM schemes}

Still using a cartesian grid, the positions of $\x_K$ are inspired by the counter-example of \cite{OMN09} to super-convergence for TPFA in dimension 1. These positions are presented in Figure \ref{fig:test2}, left. The rates observed on the right of the figure show that the super-convergence of HMM is lost, which seems to indicate that Theorem \ref{th-supercv-HMM-group} is relatively optimal,
i.e. that even for very simple grids, HMM is not super-convergent if some local compensations do not occur. As expected, the modified HMM method remains super-convergent for this case.

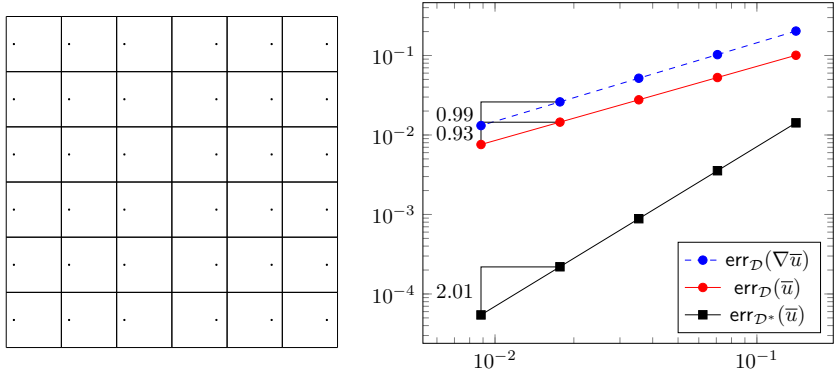
\begin{figure}[htb]
\begin{center}
\begin{tabular}{lr}
\raisebox{1em}{\resizebox{0.35\linewidth}{!}{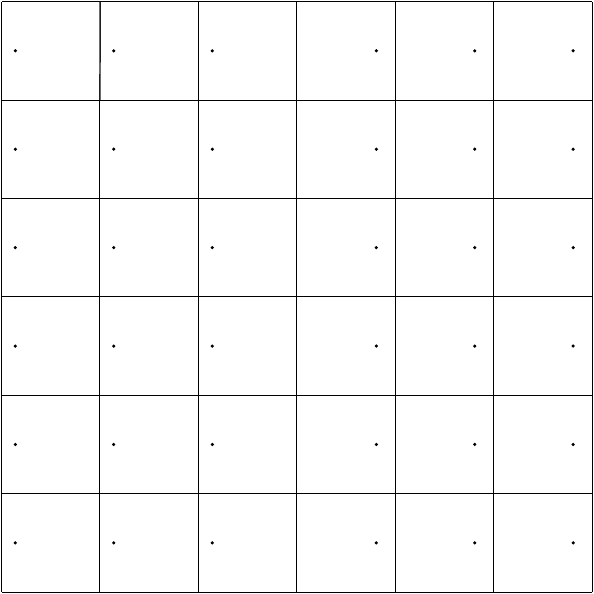}}
&
\resizebox{0.5\linewidth}{!}{%
	 \begin{tikzpicture}[scale=1]
			 \begin{loglogaxis}[
					 xmin = .00530400000000000000, 
					 legend style = {
					   legend pos = south east
					 }
				 ]
		     \addplot[dashed,mark=*,blue,mark options={solid}] table[x=meshsize,y={create col/linear regression={y=erL2_gradp}}] {HMM.test2.loss_supercv.dat}
		     coordinate [pos=0.75] (A)
		     coordinate [pos=1.00] (B);
		     \xdef\slopeb{\pgfplotstableregressiona}
		     \draw (A) -| (B) node[pos=0.75,anchor=east] {\pgfmathprintnumber{\slopeb}};
		     \addplot[mark=*,red,mark options={solid}] table[x=meshsize,y={create col/linear regression={y=erL2_p}}] {HMM.test2.loss_supercv.dat}
		     coordinate [pos=0.75] (A)
		     coordinate [pos=1.00] (B);
		     \xdef\slopea{\pgfplotstableregressiona}
		     \draw (A) -| (B) node[pos=0.75,anchor=east] {\pgfmathprintnumber{\slopea}};
		     \addplot[mark=square*,black,mark options={solid}] table[x=meshsize,y={create col/linear regression={y=erL2_p}}] {modifiedHMM.test2.supercv.dat}
		     coordinate [pos=0.75] (A)
		     coordinate [pos=1.00] (B);
		     \xdef\slopec{\pgfplotstableregressiona}
		     \draw (A) -| (B) node[pos=0.75,anchor=east] {\pgfmathprintnumber{\slopec}};
		     \legend{$\err_\disc(\nabla\bu)$,$\err_\disc(\bu)$,$\err_\discs(\bu)$};
		   \end{loglogaxis}
		 \end{tikzpicture}
}
\end{tabular}
\caption{Test 2: position of the points $\x_K$ (left), and rates of convergence (right) for the HMM and modified HMM methods.}
\label{fig:test2}
\end{center}
\end{figure}

\subsection{TPFA finite volumes on triangles}\label{sec:num.TPFA}

We illustrate here the result of Theorem \ref{th:super_cv.tpfa}, considering
three families of triangulations corresponding to the classical TPFA triangulations
as in Definition \ref{def:classical.tpfa.tri}.
Many previous numerical tests (see, e.g., \cite{dom-05-fin,GagneuxMadaune})
have numerically demonstrated the super-convergence
of TPFA on such meshes but, to our knowledge, no complete rigorous proof of this phenomenon has been 
provided so far.
All numerical results show a clear order 2 rate of convergence, confirming
Theorem \ref{th:super_cv.tpfa}.

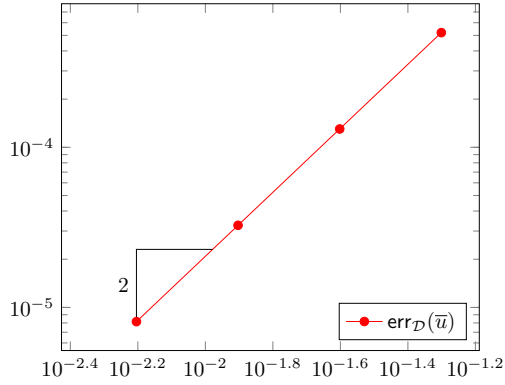
\begin{figure}
\begin{center}
\resizebox{.55\linewidth}{!}{
 \begin{tikzpicture}[scale=1]
			 \begin{loglogaxis}[
					 xmin = .00375000000000000000, 
					 legend style = {
					   legend pos = south east
					 }
				 ]
		     \addplot[mark=*,red,mark options={solid}] table[x=meshsize,y={create col/linear regression={y=erL2_p}}] {TPFA.subd.dat}
		     coordinate [pos=0.75] (A)
		     coordinate [pos=1.00] (B);
		     \xdef\slopea{\pgfplotstableregressiona}
		     \draw (A) -| (B) node[pos=0.75,anchor=east] {\pgfmathprintnumber{\slopea}};
		     \legend{$\err_\disc(\bu)$};
		   \end{loglogaxis}
		 \end{tikzpicture}
}
\end{center}
\label{fig:subd_trianglespattern}
\caption{$L^2$ rate of convergence of TPFA on the family of meshes made of subdivisions of the
initial triangulation in Figure \ref{fig:classical_subd} (left).}
\end{figure}

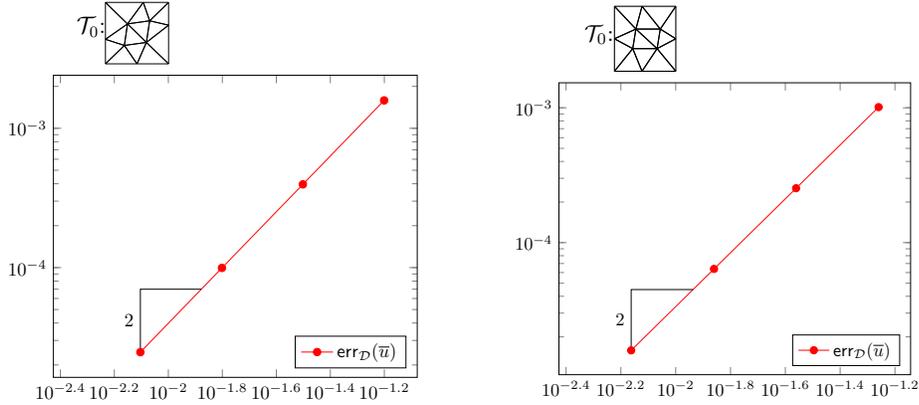
\begin{figure}[htb]
\begin{center}
\begin{tabular}{l@{\qquad}r}
\raisebox{6em}{\begin{tabular}{l}
{\hspace*{3em}\resizebox{0.1\linewidth}{!}{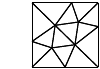}}\\
\resizebox{.45\linewidth}{!}{
 \begin{tikzpicture}[scale=1]
			 \begin{loglogaxis}[
					 xmin = .00375000000000000000, 
					 legend style = {
					   legend pos = south east
					 }
				 ]
		     \addplot[mark=*,red,mark options={solid}] table[x=meshsize,y={create col/linear regression={y=erL2_p}}] {TPFA.sym.dat}
		     coordinate [pos=0.75] (A)
		     coordinate [pos=1.00] (B);
		     \xdef\slopea{\pgfplotstableregressiona}
		     \draw (A) -| (B) node[pos=0.75,anchor=east] {\pgfmathprintnumber{\slopea}};
		     \legend{$\err_\disc(\bu)$};
		   \end{loglogaxis}
		 \end{tikzpicture}
}
\end{tabular}}
&
\raisebox{6em}{\begin{tabular}{l}
{\hspace*{3em}\resizebox{0.1\linewidth}{!}{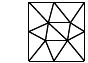}}\\
\resizebox{.45\linewidth}{!}{
 \begin{tikzpicture}[scale=1]
			 \begin{loglogaxis}[
					 xmin = .00375000000000000000, 
					 legend style = {
					   legend pos = south east
					 }
				 ]
		     \addplot[mark=*,red,mark options={solid}] table[x=meshsize,y={create col/linear regression={y=erL2_p}}] {TPFA.trans.dat}
		     coordinate [pos=0.75] (A)
		     coordinate [pos=1.00] (B);
		     \xdef\slopea{\pgfplotstableregressiona}
		     \draw (A) -| (B) node[pos=0.75,anchor=east] {\pgfmathprintnumber{\slopea}};
		     \legend{$\err_\disc(\bu)$};
		   \end{loglogaxis}
		 \end{tikzpicture}
}
\end{tabular}}
\end{tabular}
\end{center}
\caption{\label{fig:pattern.sym} Initial triangulations and $L^2$ rates of convergence 
of TPFA on families of meshes constructed by reproductions. Left: reproductions by
symmetry. Right: reproduction by translation.}
\end{figure}

\section{Conclusion}\label{sec:concl}

The contributions of the paper can be summarised as follows.
We  first establish an improved $L^2$ estimate for gradient schemes, in any dimension $d$,
which is more precise than the 
known ones of \cite{koala,eym-12-sma}. This estimate yields better rates for a number of gradient schemes.

Secondly, a modified HMM scheme with unconditional super-convergence (in dimension $d\le 3$) is introduced. This modified scheme uses a piecewise linear, 
instead of piecewise constant, approximation of the test functions. This approximation was introduced in \cite{CAN08},
but only as a post-processing tool. By using this approximation in the design of the modified HMM method, we create a 
method that is super-convergent for any choice of the cell points as a consequence of the improved $L^2$
estimate for gradient schemes.

The next contribution is a new $L^2$ error estimate for HMM, that involves patches of cells. When 
these patches can be chosen so that a compensation occurs, within each patch, between the cell points and the 
centers of mass, this new $L^2$ estimate provides the super-convergence of HMM. 
The numerical results show that in the absence of patches, super-convergence may fail for HMM 
schemes, but holds true for the modified HMM scheme. Moreover the super-convergence is recovered if local compensation 
occurs.

Finally, perhaps the main contribution of this work, we prove
the super-convergence of the TPFA finite volume scheme on the kinds of meshes used in 2D benchmarking of this method.
This result is a consequence of all the previous ones.
Numerical tests confirming this super-convergence are presented. 

\bigskip
\section{Appendix}\label{sec:appen}

\subsection{Gradient schemes with approximate diffusion}\label{sec:appen.modified}

Let $\disc$ be a gradient discretisation in the sense of Definition \ref{def:GD}.
As per \eqref{base.GS}, the corresponding gradient scheme for \eqref{base} is
\be\label{gs.appendix}
\begin{aligned}
\mbox{Find $u_\disc\in X_{\disc,0}$ such that, for all $v_\disc\in X_{\disc,0}$,}\\
\int_\O A\nabla_\disc u_\disc\cdot\nabla_\disc v_\disc \d\x
=\int_\O f\Pi_\disc v_\disc \d\x.
\end{aligned}
\ee
For low-order methods, it is however customary to replace $A$ with a piecewise
approximation on the mesh. More precisely, if $\polyd$ is a polytopal mesh of $\O$,
we denote by $A_\mesh$ the $L^2$ projection of $A$ on the piecewise constant (matrix-valued) functions
on $\mesh$, that is
\be\label{def:Amesh}
\forall\cell\in\mesh\,,\; A_\mesh = \frac{1}{|\cell|}\int_\cell A(\x)\d\x\mbox{ on $\cell$},
\ee
and we consider the modified gradient scheme
\be\label{gs.modified}
\begin{aligned}
\mbox{Find $\widetilde{u}_\disc\in X_{\disc,0}$ such that, for all $v_\disc\in X_{\disc,0}$,}\\
\int_\O A_\mesh\nabla_\disc \widetilde{u}_\disc\cdot\nabla_\disc v_\disc \d\x
=\int_\O f\Pi_\disc v_\disc \d\x.
\end{aligned}
\ee

The following two propositions show that, for low order methods (for which it is
expected that $\WS_\disc(\bu)=\mathcal O(h_\mesh)$), both the basic rate of convergence and the rate
of super-convergence are not degraded by considering \eqref{gs.modified} instead of
\eqref{gs.appendix}. In the following, we use the notation $\mathcal A\lesssim \mathcal B$
as a shorthand for ``$\mathcal A\le C\mathcal B$ for some $C$ depending only on $A$ and $\O$''.

\begin{proposition}\label{prop.comp.modified}
Let $\disc$ be a gradient discretisation of $\O$ and $\polyd$ a polytopal mesh of $\O$. 
Assume that \eqref{assump} holds, $A$ is Lipschitz-continuous on each
$\cell\in\mesh$, and $A_\mesh$ is defined by \eqref{def:Amesh}.
If $u_\disc$ and $\widetilde{u}_\disc$ are, respectively, the solutions to
\eqref{gs.appendix} and \eqref{gs.modified}, then
\be\label{prop.comp.0}
\norm{\Pi_\disc \widetilde{u}_\disc-\Pi_\disc u_\disc}{L^2(\O)}+
\norm{\nabla_\disc \widetilde{u}_\disc-\nabla_\disc u_\disc}{L^2(\O)^d}\lesssim
h_\mesh\norm{f}{L^2(\O)}.
\ee
As a consequence,
\be\label{prop.comp.1}
\norm{\Pi_\disc \widetilde{u}_\disc-\bu}{L^2(\O)}+
\norm{\nabla_\disc \widetilde{u}_\disc-\nabla\bu}{L^2(\O)^d}\lesssim
\WS_\disc(\bu)+h_\mesh\norm{f}{L^2(\O)}.
\ee
\end{proposition}

\begin{proof}
We first notice that, once it is known that $A$ is Lipschitz-continuous on each
$\cell\in\mesh$, the maximum of the Lipschitz constants of $(A_{|\cell})_{\cell\in\mesh}$
is actually independent of $\mesh$. This entails
\be\label{est.lip.A}
\norm{A-A_\mesh}{L^\infty(\O)}\lesssim h_\mesh.
\ee
We have, for $v_\disc\in X_{\disc,0}$,
\[
\int_\O A\nabla_\disc\widetilde{u}_\disc\cdot\nabla_\disc v_\disc\d\x
=\int_\O f\Pi_\disc v_\disc\d\x+\int_\O (A-A_\mesh)\nabla_\disc\widetilde{u}_\disc\cdot
\nabla_\disc v_\disc\d\x.
\]
Subtracting the gradient scheme \eqref{gs.appendix} and using \eqref{est.lip.A}, we infer
\begin{align}
\int_\O A(\nabla_\disc\widetilde{u}_\disc-u_\disc)\cdot\nabla_\disc v_\disc\d\x
&=\int_\O (A-A_\mesh)\nabla_\disc\widetilde{u}_\disc\cdot
\nabla_\disc v_\disc\d\x\label{comp.0}\\
&\lesssim h_\mesh \norm{\nabla_\disc \widetilde{u}_\disc}{L^2(\O)^d}\norm{\nabla_\disc v_\disc}{L^2(\O)^d}.
\label{comp.1}
\end{align}
It is clear that $\widetilde{u}_\disc$ still satisfies the stability property \eqref{est.stability}
(to verify this, take $v_\disc=\widetilde{u}_\disc$ in \eqref{gs.modified}, use the definition \eqref{def.CD} of
$C_\disc$, and the fact that $A_\mesh$ is uniformly coercive with the same coercivity
constant as $A$). Hence, choosing $v_\disc=\widetilde{u}_\disc-u_\disc$ in \eqref{comp.1},
\[
\norm{\nabla_\disc\widetilde{u}_\disc-\nabla_\disc u_\disc}{L^2(\O)^d}
\lesssim h_\mesh \norm{f}{L^2(\O)}.
\]
The proof of \eqref{prop.comp.0} is complete by recalling the definition \eqref{def.CD} of $C_\disc$.
The estimate \eqref{prop.comp.1} follows from a triangle inequality (introducing $\Pi_\disc v_\disc$ and $\nabla_\disc v_\disc$), \eqref{prop.comp.0}, and the estimate \eqref{est.error.gs}
in Theorem \ref{th:error.est.PDE}.
\end{proof}

\begin{proposition}
Under the assumptions of Proposition \ref{prop.comp.modified}, let us
moreover suppose that the $H^2$ regularity property \eqref{assump.H2} holds.
We also assume that, for all $\phi\in H^1_0(\O)\cap H^2(\O)$,
\be\label{est.WS}
\WS_\disc(\phi)\lesssim \left(\norm{A\nabla\phi}{L^2(\O)^d}+\norm{\phi}{H^2(\O)}\right)h_\mesh.
\ee
Let $u_\disc$ and $\widetilde{u}_\disc$ be, respectively, the solutions to \eqref{gs.appendix} and
\eqref{gs.modified}. Then,
\be\label{est.supercv.modified}
\norm{\Pi_\disc u_\disc-\Pi_\disc \widetilde{u}_\disc}{L^2(\O)}\lesssim h_\mesh^2\norm{f}{L^2(\O)}.
\ee
\end{proposition}

\begin{proof}
Let $g\in L^2(\O)$, $\varphi_g$ be the solution to \eqref{base} with $g$ instead of $f$,
and $\varphi_{g,\disc}$ be the solution to \eqref{gs.appendix} with $g$ instead of $f$.
By \eqref{est.error.gs}, \eqref{est.WS} and \eqref{assump.H2},
\[
\norm{\nabla_\disc \varphi_{g,\disc} -\nabla\varphi_g}{L^2(\O)^d}\lesssim h_\mesh\norm{g}{L^2(\O)}.
\]
Using $v_\disc=\varphi_{g,\disc}$ in \eqref{comp.0} and recalling \eqref{est.lip.A} therefore leads to
\begin{align}
\int_\O g&\Pi_\disc(\widetilde{u}_\disc-u_\disc)\d\x\nonumber\\
={}&\int_\O (A-A_\mesh)\nabla_\disc\widetilde{u}_\disc\cdot\nabla_\disc \varphi_{g,\disc}\d\x\nonumber\\
={}&\int_\O (A-A_\mesh)\nabla_\disc\widetilde{u}_\disc\cdot(\nabla_\disc \varphi_{g,\disc}-\nabla \varphi_{g})\d\x
+\int_\O (A-A_\mesh)\nabla_\disc\widetilde{u}_\disc\cdot\nabla \varphi_{g}\d\x\nonumber\\
\lesssim{}& h_\mesh^2 \norm{\nabla_\disc \widetilde{u}_\disc}{L^2(\O)^d} \norm{g}{L^2(\O)}
+\int_\O (A-A_\mesh)(\nabla_\disc\widetilde{u}_\disc-\nabla\bu)\cdot\nabla \varphi_{g}\d\x\nonumber\\
&+\int_\O (A-A_\mesh)\nabla\bu\cdot\nabla \varphi_{g}\d\x\nonumber\\
\lesssim{}& h_\mesh^2 \norm{f}{L^2(\O)} \norm{g}{L^2(\O)}
+\int_\O (A-A_\mesh)\nabla\bu\cdot\nabla \varphi_{g}\d\x.
\label{supercv.modified}
\end{align}
In the last line, we have used the standard stability
property $\norm{\nabla \varphi_{g}}{L^2(\O)^d} \lesssim\norm{g}{L^2(\O)}$, and
Proposition \ref{prop.comp.modified} along with \eqref{est.WS} and \eqref{assump.H2}.
We now estimate the last term in \eqref{supercv.modified}. By definition of $A_\mesh$,
\begin{align*}
\int_\O (A-A_\mesh)\nabla\bu\cdot\nabla \varphi_{g}\d\x
={}&\sum_{i,j=1}^d \int_\O (A-A_\mesh)_{i,j} \partial_j \bu\partial_i \varphi_{g}\d\x\\
={}&\sum_{i,j=1}^d \int_\O (A-A_\mesh)_{i,j} \left[\partial_j \bu\partial_i \varphi_{g}
-\pi_\mesh(\partial_j\bu\partial_i\varphi_{g})\right]\d\x,
\end{align*}
where $\pi_\mesh$ denotes projection on piecewise constant functions on $\mesh$.
By classical estimates (see e.g. \cite[Lemma B.7]{koala}),
\[
\norm{\partial_j \bu\partial_i \varphi_{g}-\pi_\mesh(\partial_j\bu\partial_i\varphi_{g})}{L^1(\O)}
\lesssim  \norm{\partial_j \bu\partial_i \varphi_{g}}{W^{1,1}(\O)}h_\mesh
\lesssim  \norm{\bu}{H^2(\O)}\norm{\varphi_{g}}{H^2(\O)}h_\mesh.
\]
Hence, using \eqref{est.lip.A} and the $H^2$ regularity property \eqref{assump.H2},
\[
\int_\O (A-A_\mesh)\nabla\bu\cdot\nabla\varphi_{g}\d\x
\lesssim h_\mesh^2\norm{f}{L^2(\O)}\norm{g}{L^2(\O)}.
\]
Plugging this estimate into \eqref{supercv.modified} and taking the supremum of the resulting
inequality over $g\in L^2(\O)$ of norm $1$ concludes the proof of \eqref{est.supercv.modified}.
\end{proof}

\begin{remark} Similar results could be obtained for higher-order methods, by considering
as $A_\mesh$ the $L^2$ projection on piecewise polynomial functions on $\mesh$.
\end{remark}

\subsection{Technical results}\label{sec:appen.tech}

The following definition appears in \cite{DEH15,koala}, in a slightly more general
context (here, it is restricted to the Hilbertian case). 

\begin{definition}[$\mathbb{P}_1$-exact gradient reconstruction]\label{def:linex}~
Let $U$ be a boun\-ded subset of $\R^d$ with non-zero measure,
and let $S = (\x_i)_{i\in I}\subset \R^d$ be a finite family of points.
A $\mathbb{P}_1$-exact gradient
reconstruction on $U$ upon $S$ is a family $\gr=(\gr^i)_{i\in I}$ of functions in
$L^2(U)^d$ such that, for any affine mapping $\ell:\R^d\rightarrow \R$ and a.e. $\x\in U$,
\[
\sum_{i\in I}\ell(\x_i)\gr^i(\x)=\nabla \ell.
\]
The norm of $\gr$ is defined by 
\begin{equation}
 \Vert \gr\Vert = 
\diam(U)|U|^{-\frac{1}{2}}\left\Vert \sum_{i\in I}|\gr^i|\right\Vert_{L^2(U)}.
\label{def:norlingrad}\end{equation}
For any family $\xi=(\xi_i)_{i\in I}$ of real numbers, define $\gr\xi=\sum_{i\in I}\xi_i\gr^i
\in L^2(U)^d$, and notice that
\[
\norm{\gr\xi}{L^2(U)^d}\le \diam(U)^{-1}|U|^{\frac{1}{2}}\norm{\gr}{}\max_{i\in I}|\xi_i|.
\]
\end{definition}

The following lemma is a specific case of \cite[Lemma A.3]{koala}.
The polynomial $L_\phi$ in this lemma is similar to
an averaged Taylor polynomial as in \cite{BS08}.

\begin{lemma}[Approximation of $H^{2}$ functions by affine functions]\label{lem:app.H2.lin}
Let $d\le 3$ and assume that $V\subset \R^d$ is bounded
and star-shaped with respect to all points in a ball $B$. 
Choose $\theta \ge \diam(V)/\diam(B)$ and $\phi \in H^{2}(V)\cap C(\overline{V})$. 

Then, there exists $\ctel{cst:estsdwdp}>0$, depending only on $d$ and $\theta$, 
and an affine function $L_\phi~:~V\to \R$ such that
\be\label{lem:app.H2.lin.est1}
 \sup_{\x\in \overline{V}} |\phi(\x) - L_\phi(\x)| \le \cter{cst:estsdwdp} \diam(V)^2
|V|^{-\frac 1 2} \norm{\phi}{H^2(V)}
\ee
and
\be\label{lem:app.H2.lin.est2}
 \norm{\nabla L_\phi -\nabla\phi}{L^2(V)^d}\le  \cter{cst:estsdwdp} \diam(V)  \norm{\phi}{H^2(V)}.
\ee
\end{lemma}

The next lemma estimates the difference between the averages of a function 
on two neighbouring sets.

\begin{lemma}\label{lem.est.moy} Let $U$, $V$ and $O$ be open sets of $\R^d$
such that, for all $(\x,\y)\in U\times V$, $[\x,\y]\subset O$.
There exists $\ctel{cst:moy}$ only depending on $d$ such that, for all
$\phi\in W^{1,1}(O)$,
\[
\left|\frac{1}{|U|}\int_U \phi(\x)\dx-\frac{1}{|V|}\int_V \phi(\x)\dx\right|
\le \frac{\cter{cst:moy}\diam(O)^{d+1}}{|U|\,|V|}\int_O |\nabla \phi(\x)|\dx.
\]
\end{lemma}

\begin{proof} 
Since $C^\infty(O)\cap W^{1,1}(O)$ is dense in $W^{1,1}(O)$, we can assume that
$\phi\in C^\infty(O)\cap W^{1,1}(O)$. We then write, by Taylor's expansion,
$\phi(\x)-\phi(\y)=\int_0^1 \nabla \phi(t\x+(1-t)\y)\cdot(\x-\y)\d t$ for $(\x,\y)\in U\times V$,
and thus
\begin{eqnarray}
\left|\frac{1}{|U|}\int_U \right.\lefteqn{\phi(\x)\dx-\left.
\frac{1}{|V|}\int_V \phi(\x)\dx\right|}&&\nonumber\\
&&\le \frac{\diam(O)}{|U|\,|V|}\int_U \int_V \int_0^1 |\nabla \phi(t\x+(1-t)\y)|\d t\d\y\d\x.
\label{convex0}\end{eqnarray}
Let us fix $\y\in V$ and apply the change of variable $\x\in U\to \z=t\x+(1-t)\y\in O$.
This gives
\begin{equation}
\int_U \int_V \int_0^1 |\nabla \phi(t\x+(1-t)\y)|\d t\d\x\d\y\le
\int_{O}|\nabla \phi(\mathbi{z})|\int_V \int_{I(\z,\y)} t^{-d}\d t\d\y\d\z
\label{convex1}\end{equation}
where $I(\z,\y)=\{t\in [0,1]\;|\;\exists \x\in U\,,\;t\x+(1-t)\y=\z\}$.
As in Step 1 of the proof of \cite[Lemma 6.6]{dro-08-stu}, we see that
\[
\int_V \int_{I(\z,\y)} t^{-d}\d t\d\y\le
\frac{\cter{ggg}}{d-1}\diam(O)^d
\]
where $\ctel{ggg}$ is the surface of the unit sphere in $\R^d$.
The proof is complete by substituting this inequality into \refe{convex1} and 
plugging the result in \refe{convex0}. \end{proof}

The existence of the functions $w_\cell$ mentioned in the following lemma
has been first established in \cite[Lemma A.1]{dro-10-uni}. We provide here some
additional estimates on these functions.

\begin{lemma}\label{lem.w.xK}
Let $d\le 3$ and $\polyd$ be a polytopal mesh of $\O$ in the sense of Definition \ref{def:polymesh}.
There exist affine functions $(w_\cell)_{\cell\in\mesh}$, and $\ctel{cst:bound.w}$
depending only on $d$ and an upper bound of $\theta_\polyd$, such that, for all
$\cell\in\mesh$,
\[
\int_\cell w_\cell(\x)\d\x=|\cell|\,,\quad
\int_\cell \x w_\cell(\x)\d\x = |\cell|\x_\cell\,,\quad
\norm{w_\cell}{L^\infty(\cell)}\le \cter{cst:bound.w}\,,
\]
and, for all $\phi\in H^2(\cell)$,
\be\label{est.phi.w}
\left|\phi(\x_\cell)-\frac{1}{|\cell|}\int_\cell \phi(\x)w_\cell(\x)\d\x\right|
\le \cter{cst:bound.w}h_\cell^2 |K|^{-\frac{1}{2}}\norm{\phi}{H^2(\cell)}.
\ee
\end{lemma}

\begin{proof} 
Consider the function given by $w_\cell(\x)=1+\xi\cdot(\x-\overline{\x}_\cell)$,
where $\xi$ is the vector such that $J_\cell \xi = |\cell|(\x_\cell-\overline{\x}_\cell)$,
with $J_\cell$ the $d\times d$ matrix given by
\[
J_\cell = \int_\cell (\x-\overline{\x}_\cell)(\x-\overline{\x}_\cell)^T\d\x.
\]
Let us now establish the estimate on $w_\cell$. We refer to Figure \ref{fig:est-w}
for an illustration of the reasoning. Up to a change of coordinate
system, we can assume that $\overline{\x}_\cell$ lies on the hyperplane
$H_0=\{\x\,:\,x_d=0\}$, and that $\xi$ is orthogonal to $H_0$
and points towards the direction $x_d>0$. By definition of $\theta_\polyd$,
$\cell$ contains a cube $Q_\cell$ centered at $\x_\cell$ and of length $\ctel{cst:height}h_\cell$,
where $\cter{cst:height}$ only depends on $d$ and an upper bound of
$\theta_\polyd$.
Let $R_\pm$ be the upper and lower thirds of $Q_\cell$, that is
$R_+=\{\x\in Q_\cell\,:\,(\x-\x_\cell)_d> \frac{\cter{cst:height}h_\cell}{6}\}$
and $R_-=\{\x\in Q_\cell\,:\,(\x-\x_\cell)_d<- \frac{\cter{cst:height}h_\cell}{6}\}$.

\begin{figure}[htb]
\begin{center}
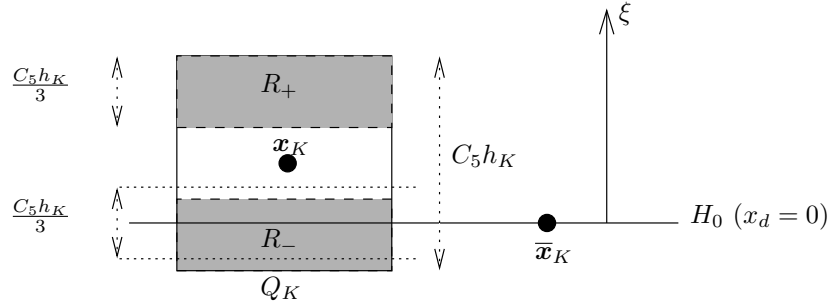
\caption{Illustration of the proof of Lemma \ref{lem.w.xK}.}
\label{fig:est-w}
\end{center}
\end{figure}

Since $Q_\cell\backslash (R_+\cup R_-)$ has width $\frac{\cter{cst:height}h_\cell}{3}$,
one of the regions $R_+$ or $R_-$ (let us assume $R_+$),
must lie entirely outside the band of width $\frac{\cter{cst:height}h_\cell}{3}$
around $x_d=0$. If $\x\in R_+$, we then have
$|(\x-\overline{\x}_\cell)\cdot\xi|=\dist(\x,H_0)|\xi|\ge \frac{\cter{cst:height}h_\cell}{3}|\xi|$.
Hence,
\begin{multline*}
|\cell|h_\cell|\xi|\ge |\cell|(\x_\cell-\overline{\x}_\cell)\cdot \xi = J_\cell \xi\cdot\xi
=\int_\cell ((\x-\overline{\x}_\cell)\cdot\xi)^2\d\x\\
\ge \int_{R_+} ((\x-\overline{\x}_\cell)\cdot\xi)^2\d\x
\ge |R_+| \frac{\cter{cst:height}^2h_\cell^2}{9}|\xi|^2.
\end{multline*}
We have $|R_+|=\frac{\cter{cst:height}^dh_\cell^d}{3}\ge \cter{cst:height}^d\cter{cst:height2}|\cell|$,
where $\ctel{cst:height2}$ only depends on $d$ and an upper bound of $\theta_\polyd$.
Hence, $|\xi|\le 9\cter{cst:height}^{-2-d}\cter{cst:height2}^{-1}h_\cell^{-1}$ and,
for all $\x\in \cell$,
\be\label{est.w.1}
|w_\cell(\x)|\le 1 + h_\cell |\xi|\le 1+9\cter{cst:height}^{-2-d}\cter{cst:height2}^{-1}.
\ee
This concludes the proof of the estimate on $\norm{w_\cell}{L^\infty(\cell)}$.

To prove \eqref{est.phi.w}, we use Lemma \ref{lem:app.H2.lin} with $V=\cell$
(since $\cell$ is star-shaped and $d\le 3$, we have $\phi\in C(\overline{K})$).
A triangle inequality gives
\begin{align}
\Big|\phi(\x_\cell)-&\frac{1}{|\cell|}\int_\cell\phi(\x)w_\cell(\x)\d\x\Big|\nonumber\\
\le{}&|\phi(\x_\cell)-L_\phi(\x_\cell)|
+\Big|L_\phi(\x_\cell)-\frac{1}{|\cell|}\int_\cell L_\phi(\x)w_\cell(\x)\d\x\Big|\nonumber\\
&+\frac{1}{|\cell|}\int_\cell\left|L_\phi(\x)-\phi(\x)\right|w_\cell(\x)\d\x.
\label{est.phi.w.1}
\end{align}
We have $L_\phi(\x)=L_\phi(\x_\cell)+\nabla L_\phi \cdot(\x-\x_\cell)$ and thus
\[
\frac{1}{|\cell|}\int_\cell L_\phi(\x)w_\cell(\x)\d\x
=L_\phi(\x_\cell)+\frac{1}{|\cell|}\nabla L_\phi\cdot\int_\cell (\x-\x_\cell)w_\cell(\x)\d\x
=L_\phi(\x_\cell).
\]
Hence, \eqref{est.phi.w.1} and the properties of $L_\phi$ give
\[
\Big|\phi(\x_\cell)-\frac{1}{|\cell|}\int_\cell\phi(\x)w_\cell(\x)\d\x\Big|
\le\cter{cst:est.w.2}h_\cell^2|\cell|^{-\frac{1}{2}}\norm{\phi}{H^2(\cell)}
(1+\norm{w_\cell}{L^\infty(\cell)}),
\]
where $\ctel{cst:est.w.2}$ only depends on $d$ and an upper bound of $\theta_\polyd$.
The estimate \eqref{est.phi.w} is complete by using \eqref{est.w.1}. \end{proof}

The following result was used in the proof of the super-convergence of HMM schemes
(Theorem \ref{th-supercv-HMM-group}),
to estimate a residual on the part of the domain not covered by the patches. It is
a $W^{1,1}$ hyperplanar version of Ilin's inequality in $H^1$ \cite{OR79}.
\begin{lemma}\label{lemma.strip}
Let $\O$ be an open set with a Lipschitz boundary,
$H$ be an hyperplane, $\rho\in (0,\diam(\O))$ and $\strip_H(\rho)=\{x\in\O\,:\,\dist(\x,H)\le \rho\}$.
Then there exists $C$ depending only on $\O$ such that, for all $\phi\in W^{1,1}(\O)$,
\be\label{ineq.strip}
\norm{\phi}{L^1(\strip_H(\rho))}\le C \rho \norm{\phi}{W^{1,1}(\O)}.
\ee
\end{lemma}
\begin{proof}
Using an affine transformation, an extension operator $W^{1,1}(\O)\to W^{1,1}(\R^d)$
(whose norm only impacts $C$ in \eqref{ineq.strip}) and the density of smooth functions
in $W^{1,1}(\R^d)$, we can
assume that $H=\R^{d-1}\times \{0\}$, that $\O$ is replaced with $\R^{d-1}\times (-\diam(\O),\diam(\O))$,
and that $\phi$ is smooth.
Then, for $\x\in \R^{d-1}$, $y\in [-\rho,\rho]$ and $z\in (-\diam(\O),\diam(\O))$,
\[
|\phi(\x,y)|=\left|\phi(\x,z)+\int_z^y \partial_d \phi(\x,s)\d s \right|
\le |\phi(\x,z)|+\int_{-\diam(\O)}^{\diam(\O)} |\partial_d \phi(\x,s)|\d s.
\]
Integrate over $y\in [-\rho,\rho]$, $z\in (-\diam(\O),\diam(\O))$ and $\x\in \R^{d-1}$:
\begin{multline*}
2\diam(\O)\int_{\R^{d-1}}\int_{-\rho}^\rho |\phi(\x,y)|\d y\d \x\\
\le
 2\rho \int_{\R^{d-1}}\int_{-\diam(\O)}^{\diam(\O)}|\phi(\x,z)|\d z\d \x+4\diam(\O)\rho \int_{\R^{d-1}}\int_{-\diam(\O)}^{\diam(\O)} |\partial_d \phi(\x,s)|\d s\d z.
\end{multline*}
Estimate \eqref{ineq.strip} follows by dividing throughout by $2\diam(\O)$.
\end{proof}

The last technical lemma of this appendix shows that the average, over a triangulation
of a set $Q$, of the differences between the circumcenters
and the centers of mass of the triangles can be computed using only the vertices
of the triangulation on $\partial Q$. This lemma is useful to find patches,
in the proof of Theorem \ref{th:super_cv.tpfa}, over which this average vanishes.

\begin{lemma}\label{lem:compensation.triangles}
Let $Q$ be a polygonal subset of $\R^2$, with center of mass $\overline{\x}_Q$,
and let $\polyd^Q=(\mesh^Q,\edges^Q,\centers^Q)$ be a 
conforming triangulation of $Q$ into triangles. For each $T\in\mesh^Q$ we denote by $\c_T$ the circumcenter
of $T$ and by $\overline{\x}_T$ the center of mass of $T$. If $\sigma\in\edges^Q$ is
an edge of the triangulation, we denote by $\vertex^\sigma_1$ and $\vertex^\sigma_2$
the two endpoints of $\edge$. Then,
\be\label{form.compens.triangles}
\sum_{T\in\mesh^Q} |T|(\c_T-\overline{\x}_T)=\sum_{\edge\in\edges^Q_{\rm ext}}|\edge|\frac{|\vertex^\edge_1-\overline{\x}_Q|^2+|\vertex^\edge_2-\overline{\x}_Q|^2}{4} \bfn_{Q,\edge},
\ee
where $\bfn_{Q,\edge}$ is the outer normal to $Q$ on $\edge$. 
\end{lemma}

\begin{proof}

Let us first establish the following formula, for any triangle $T$:
\be\label{form:cT}
|T|\c_T = \sum_{\edge\in\edges_T} |\edge|\frac{|\vertex^\edge_1|^2+|\vertex^\edge_2|^2}{4}\bfn_{T,\edge}.
\ee

\begin{figure}[htb]
\begin{center}
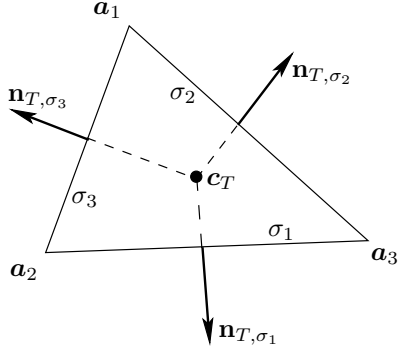
\caption{Notations within a triangle $T$.}
\label{fig:cT}
\end{center}
\end{figure}

We use the notations in Figure \ref{fig:cT}. Any vector $\xi\in\R^2$ can be written 
\[
\xi = -\frac{1}{2|T|}\left(|\edge_3|[\xi\cdot (\a_3-\a_1)]\bfn_{T,\edge_3}
+|\edge_2|[\xi\cdot(\a_2-\a_1)]\bfn_{T,\edge_2}\right)
\]
(take the dot product of each side this inequality
with the two linearly independent vectors $\a_3-\a_1$ and $\a_2-\a_1$).
We apply this relation to $\xi=\c_T$ and use the characterisations
\[
\left(\c_T-\frac{\a_1+\a_3}{2}\right)\cdot(\a_3-\a_1)=0\mbox{ and }
\left(\c_T-\frac{\a_1+\a_2}{2}\right)\cdot(\a_2-\a_1)=0
\]
of $\c_T$ to obtain
\begin{align*}
-2|T|\c_T ={}& |\edge_3|[\c_T\cdot (\a_3-\a_1)]\bfn_{T,\edge_3}
+|\edge_2|[\c_T\cdot(\a_2-\a_1)]\bfn_{T,\edge_2}\\
={}&|\edge_3|\left[\left(\frac{\a_1+\a_3}{2}\right)
\cdot (\a_3-\a_1)\right]\bfn_{T,\edge_3}\\
&+|\edge_2|\left[\left(\frac{\a_1+\a_2}{2}\right)\cdot(\a_2-\a_1)\right]
\bfn_{T,\edge_2}\\
={}&\frac{1}{2}\left(|\edge_3|\left[|\a_3|^2-|\a_1|^2\right]\bfn_{T,\edge_3}
+|\edge_2|\left[|\a_2|^2-|\a_1|^2\right]\bfn_{T,\edge_2}\right).
\end{align*}
Since $|\edge_1|\bfn_{T,\edge_1}+|\edge_2|\bfn_{T,\edge_2}+|\edge_3|\bfn_{T,\edge_3}=0$,
we infer
\begin{align*}
4|T|\c_T ={}&|\a_1|^2\big[|\edge_2|\bfn_{T,\edge_2}+|\edge_3|\bfn_{T,\edge_3}\big]
+|\a_2|^2\big[|\edge_1|\bfn_{T,\edge_1}+|\edge_3|\bfn_{T,\edge_3}\big]\\
&+|\a_3|^2\big[|\edge_1|\bfn_{T,\edge_1}+|\edge_2|\bfn_{T,\edge_2}\big].
\end{align*}
Gathering this sum by edges contributions $|\edge_i|\bfn_{T,\edge_i}$ concludes
the proof of \eqref{form:cT}.

\medskip

The proof of \eqref{form.compens.triangles} is now trivial. We have
$|Q|\overline{\x}_Q=\sum_{T\in\mesh^Q}|T|\overline{\x}_T$ and, without
loss of generality, we can assume that this quantity is equal to $0$
(we translate $Q$ so that its center of mass is $0$).
By summing \eqref{form:cT} over $T\in\mesh^Q$ and gathering the right-hand side
by edges, we find
\[
\sum_{T\in\mesh^Q}|T|\c_T = \sum_{\edge\in\edges^Q_{\rm int}} |\edge|
\frac{|\vertex^\edge_1|^2+|\vertex^\edge_2|^2}{4}
(\bfn_{T,\edge}+\bfn_{T',\edge})
+\sum_{\edge\in\edges^Q_{\rm ext}} |\edge|
\frac{|\vertex^\edge_1|^2+|\vertex^\edge_2|^2}{4}\bfn_{Q,\edge}.
\]
In the first sum, $T$ and $T'$ are the triangles on each side of $\edge$,
and thus $\bfn_{T,\edge}+\bfn_{T',\edge}=0$. The proof of \eqref{form.compens.triangles}
is complete. \end{proof}

\subsection{Implementation and fluxes of the HMM and modified HMM methods}\label{sec:impl.hmm}

We give here some elements for implementing the HMM method (the gradient scheme
\eqref{base.GS} based on the gradient discretisation in Definition \ref{def.HMM.GD}),
and the modified HMM method of Section \ref{sec:mod.hmm}, which also
leads us to discuss
their interpretation as finite volume methods for appropriate choices of fluxes.

\subsubsection{HMM method}

The following is drawn from \cite{sushi,dro-10-uni}, and solely recalled for
ease of reference. The fluxes $(F_{\cell,\edge}(u))_{\cell\in\mesh\,,\;\edge\in\edgescv}$
of the HMM method are defined, for $u\in X_{\disc,0}$, by
\begin{multline}\label{def.flux}
\forall \cell\in\mesh\,,\;\forall v=(v_\cell,(v_{\edge})_{\edge\in\edgescv})\,,\\
\sum_{\edge\in\edgescv}F_{\cell,\edge}(u)(v_\cell-v_\edge)=\int_K A(\x)\nabla_\disc u(\x)
\nabla_\disc v(\x)\d\x
\end{multline}
(note that, on $\cell$, $\nabla_\disc v$ only depends on $(v_\cell,(v_{\edge})_{\edge\in\edgescv})$). Then, $u\in X_{\disc,0}$ is a solution to the HMM scheme
if and only if the following balance and conservativity of fluxes are satisfied:
\begin{align}
\label{balance.fluxes}
&\forall \cell\in \mesh\,,\;\sum_{\edge\in\edgescv}F_{K,\edge}(u)=\int_K f(\x)\d\x\,,\\
\label{conservativity.fluxes}
&\forall \edge\in\edgesint\mbox{, if $\mesh_\edge=\{K,L\}$ then }F_{K,\edge}(u)+F_{L,\edge}(u)=0.
\end{align}
These equations are respectively obtained by taking, in \eqref{base.GS}, a test function $v$ that
is equal to $1$ on $\cell$ and zero at all other degrees of freedom, and a test function $v$
that is equal to $1$ on $\edge$ and zero at all other degrees of freedom. The HMM method is
best implemented through \eqref{balance.fluxes}--\eqref{conservativity.fluxes}, once
a practical local formula for the fluxes is obtained.

\begin{remark}[Formula for the fluxes]
Thanks to \eqref{average.grad}, assuming that $A$ is constant equal to $A_\cell$ on $\cell$,
we have
\be\label{re.flux}
\int_\cell A(\x)\nabla_\disc u(\x)\nabla_\disc v(\x)\d\x=
|\cell|A_\cell\grad_\cell u \cdot\grad_\cell v + R_\cell(v)^T \mathbb{B}_\cell R_\cell(u),
\ee
where $\mathbb{B}_\cell$ is a ${\rm Card}(\edgescv)\times {\rm Card}(\edgescv)$ symmetric positive definite matrix related to $\iso_\cell$. If $\iso_\cell=\alpha_\cell{\rm Id}$
(usual choice), then $\mathbb{B}_K=\alpha_\cell^2 {\rm diag}(\frac{|\edge|}{d_{\cell,\edge}}A_\cell \bfn_{\cell,\edge}\cdot\bfn_{\cell,\edge})$. To implement the HMM method in practice,
one chooses $\mathbb{B}_K$. $\iso_K$ is only a tool for the analysis of the method.

Owing to \eqref{def.flux} and \eqref{re.flux}, the fluxes can be written
$(F_{\cell,\edge}(u))_{\edge\in\edgescv}=\mathbb{W}_\cell(u_\cell-u_\edge)_{\edge\in\edgescv}$.
Here, $\mathbb{W}_\cell$ is the square matrix of size ${\rm Card}(\edgescv)$ defined by
\[
\mathbb{W}_\cell=|\cell|\mathbb{G}_\cell^T A_\cell\mathbb{G}_\cell + \mathbb{R}_\cell^T
\mathbb{B}_\cell\mathbb{R}_\cell,
\]
where $\mathbb{G}_\cell$ is the $d\times {\rm Card}(\edgescv)$ matrix with columns 
$\frac{|\edge|}{|\cell|}\bfn_{\cell,\edge}$, and $\mathbb{R}_\cell=\mathbb{I}_\cell
-\mathbb{X}_\cell\mathbb{G}_\cell$ with $\mathbb{I}_\cell$ the ${\rm Card}(\edgescv)\times {\rm Card}(\edgescv)$ identity matrix and $\mathbb{X}_K$ the matrix
with rows $((\centeredge-\x_\cell)^T)_{\edge\in\edgescv}$.
\end{remark}

\subsubsection{Modified HMM method}

Given that the modification is only on $\Pi_\disc$ (cf. \eqref{def.Pidiscs}), as previously
mentioned the matrix of the modified HMM method is identical to the matrix of the HMM method.
The fluxes of the modified HMM method are therefore still defined by \eqref{def.flux}.
If $v$ is equal to 1 at the degree of freedom corresponding to $K$ and to zero
at all other degrees of freedom, then $\Pi_\discs v=1=\Pi_\disc v$. Hence, the
rows of the source-term corresponding to cell degrees of freedom are also unchanged
with respect to the HMM method. This means that the balance of fluxes \eqref{balance.fluxes}
remains.

The only changes, from the HMM to the modified HMM scheme, in the source-term are 
in the rows corresponding to interior edge unknowns. Taking $v$ equal to $1$ on
$\edge$ (such that $\mesh_\edge=\{K,L\}$) and zero at all other degrees of freedom, 
we have $\nabla_K v=\frac{|\edge|}{|K|}\bfn_{K,\edge}$ (and similarly for $\nabla_L v$)
and therefore
the conservativity equation \eqref{conservativity.fluxes} is modified into
\begin{multline}\label{cons.mod.hmm}
F_{K,\edge}(u)+F_{L,\edge}(u)\\
=\frac{|\edge|}{|K|}\int_K f(\x) \bfn_{K,\edge}\cdot(\x-\x_K)\d\x
+\frac{|\edge|}{|L|}\int_L f(\x) \bfn_{L,\edge}\cdot(\x-\x_L)\d\x.
\end{multline}
The fluxes of the modified HMM method are therefore no longer conservative,
and the modified HMM method is not a finite volume scheme.


\begin{remark}[Preserving the conservativity]
Two options exist to preserve the conservativity of the modified HMM method.
The first one is to re-define the fluxes by setting, for all $K\in\mesh$
and all $\edge\in\edgescv$,
\[
F^*_{K,\edge}(u)=F_{K,\edge}(u) - \frac{|\edge|}{|K|}\int_K f(\x) \bfn_{K,\edge}\cdot(\x-\x_K)\d\x.
\]
Then, from \eqref{cons.mod.hmm} we deduce that
$F^*_{K,\edge}(u)+F^*_{L,\edge}(u)=0$ for all interior edge $\edge$.
Moreover, since $\sum_{\edge\in\edgescv} |\edge|\bfn_{K,\edge}=0$, we
have $\sum_{\edge\in\edgescv}F^*_{K,\edge}(u)=\sum_{\edge\in\edgescv}F_{K,\edge}(u)
=\int_K f(\x)\d\x$, i.e. the new fluxes still satisfy the balance equation.

Another option is to modify the reconstruction $\Pi_\discs$ by taking, instead
of $\nabla_K$, a linearly exact gradient reconstruction based on the cell degrees of
freedom, and not using any edge degrees of freedom. The corresponding new modified HMM
scheme is then naturally conservative, but the source term in the balance equation is
modified (it involves $f$ in neighbouring cells).
\end{remark}

\bigskip

\thanks{Acknowledgement: This research was supported by the Australian Government through the Australian Research Council's Discovery Projects funding scheme (pro\-ject number DP170100605). The authors would like to warmly thanks the referees
for their clever remarks, which lead to an improvement of the super-convergence
result for TPFA.}

\bibliographystyle{abbrv}
\bibliography{jerome_nn}

\end{document}

%% file: fig-subd_triangles.pdf_t
\begin{picture}(0,0)%
\includegraphics{fig-subd_triangles.pdf}%
\end{picture}%
\setlength{\unitlength}{4144sp}%
\begingroup\makeatletter\ifx\SetFigFont\undefined%
\gdef\SetFigFont#1#2#3#4#5{%
  \reset@font\fontsize{#1}{#2pt}%
  \fontfamily{#3}\fontseries{#4}\fontshape{#5}%
  \selectfont}%
\fi\endgroup%
\begin{picture}(4373,4392)(-55,-7221)
\end{picture}%

%% file: fig-pattern_symmetry.pdf_t
\begin{picture}(0,0)%
\includegraphics{fig-pattern_symmetry.pdf}%
\end{picture}%
\setlength{\unitlength}{4144sp}%
\begingroup\makeatletter\ifx\SetFigFont\undefined%
\gdef\SetFigFont#1#2#3#4#5{%
  \reset@font\fontsize{#1}{#2pt}%
  \fontfamily{#3}\fontseries{#4}\fontshape{#5}%
  \selectfont}%
\fi\endgroup%
\begin{picture}(3399,3399)(5389,-5698)
\put(5536,-2761){\makebox(0,0)[lb]{\smash{{\SetFigFont{29}{34.8}{\rmdefault}{\mddefault}{\updefault}{\color[rgb]{0,0,0}\itr}%
}}}}
\put(6211,-2761){\makebox(0,0)[lb]{\smash{{\SetFigFont{29}{34.8}{\rmdefault}{\mddefault}{\updefault}{\color[rgb]{0,0,0}\itrh}%
}}}}
\put(5536,-3256){\makebox(0,0)[lb]{\smash{{\SetFigFont{29}{34.8}{\rmdefault}{\mddefault}{\updefault}{\color[rgb]{0,0,0}\itrv}%
}}}}
\put(6211,-3256){\makebox(0,0)[lb]{\smash{{\SetFigFont{29}{34.8}{\rmdefault}{\mddefault}{\updefault}{\color[rgb]{0,0,0}\itrr}%
}}}}
\put(6931,-4111){\makebox(0,0)[lb]{\smash{{\SetFigFont{29}{34.8}{\rmdefault}{\mddefault}{\updefault}{\color[rgb]{0,0,0}\itr}%
}}}}
\put(7606,-4111){\makebox(0,0)[lb]{\smash{{\SetFigFont{29}{34.8}{\rmdefault}{\mddefault}{\updefault}{\color[rgb]{0,0,0}\itrh}%
}}}}
\put(6931,-4606){\makebox(0,0)[lb]{\smash{{\SetFigFont{29}{34.8}{\rmdefault}{\mddefault}{\updefault}{\color[rgb]{0,0,0}\itrv}%
}}}}
\put(7606,-4606){\makebox(0,0)[lb]{\smash{{\SetFigFont{29}{34.8}{\rmdefault}{\mddefault}{\updefault}{\color[rgb]{0,0,0}\itrr}%
}}}}
\put(5581,-4111){\makebox(0,0)[lb]{\smash{{\SetFigFont{29}{34.8}{\rmdefault}{\mddefault}{\updefault}{\color[rgb]{0,0,0}\itr}%
}}}}
\put(6256,-4111){\makebox(0,0)[lb]{\smash{{\SetFigFont{29}{34.8}{\rmdefault}{\mddefault}{\updefault}{\color[rgb]{0,0,0}\itrh}%
}}}}
\put(5581,-5461){\makebox(0,0)[lb]{\smash{{\SetFigFont{29}{34.8}{\rmdefault}{\mddefault}{\updefault}{\color[rgb]{0,0,0}\itr}%
}}}}
\put(6256,-5461){\makebox(0,0)[lb]{\smash{{\SetFigFont{29}{34.8}{\rmdefault}{\mddefault}{\updefault}{\color[rgb]{0,0,0}\itrh}%
}}}}
\put(6931,-5461){\makebox(0,0)[lb]{\smash{{\SetFigFont{29}{34.8}{\rmdefault}{\mddefault}{\updefault}{\color[rgb]{0,0,0}\itr}%
}}}}
\put(7606,-5461){\makebox(0,0)[lb]{\smash{{\SetFigFont{29}{34.8}{\rmdefault}{\mddefault}{\updefault}{\color[rgb]{0,0,0}\itrh}%
}}}}
\put(6931,-2761){\makebox(0,0)[lb]{\smash{{\SetFigFont{29}{34.8}{\rmdefault}{\mddefault}{\updefault}{\color[rgb]{0,0,0}\itr}%
}}}}
\put(6931,-3256){\makebox(0,0)[lb]{\smash{{\SetFigFont{29}{34.8}{\rmdefault}{\mddefault}{\updefault}{\color[rgb]{0,0,0}\itrv}%
}}}}
\put(8281,-2761){\makebox(0,0)[lb]{\smash{{\SetFigFont{29}{34.8}{\rmdefault}{\mddefault}{\updefault}{\color[rgb]{0,0,0}\itr}%
}}}}
\put(8281,-3256){\makebox(0,0)[lb]{\smash{{\SetFigFont{29}{34.8}{\rmdefault}{\mddefault}{\updefault}{\color[rgb]{0,0,0}\itrv}%
}}}}
\put(5581,-4606){\makebox(0,0)[lb]{\smash{{\SetFigFont{29}{34.8}{\rmdefault}{\mddefault}{\updefault}{\color[rgb]{0,0,0}\itrv}%
}}}}
\put(6256,-4606){\makebox(0,0)[lb]{\smash{{\SetFigFont{29}{34.8}{\rmdefault}{\mddefault}{\updefault}{\color[rgb]{0,0,0}\itrr}%
}}}}
\put(7606,-2761){\makebox(0,0)[lb]{\smash{{\SetFigFont{29}{34.8}{\rmdefault}{\mddefault}{\updefault}{\color[rgb]{0,0,0}\itrh}%
}}}}
\put(7606,-3256){\makebox(0,0)[lb]{\smash{{\SetFigFont{29}{34.8}{\rmdefault}{\mddefault}{\updefault}{\color[rgb]{0,0,0}\itrr}%
}}}}
\put(8281,-4111){\makebox(0,0)[lb]{\smash{{\SetFigFont{29}{34.8}{\rmdefault}{\mddefault}{\updefault}{\color[rgb]{0,0,0}\itr}%
}}}}
\put(8281,-4606){\makebox(0,0)[lb]{\smash{{\SetFigFont{29}{34.8}{\rmdefault}{\mddefault}{\updefault}{\color[rgb]{0,0,0}\itrv}%
}}}}
\put(8281,-5461){\makebox(0,0)[lb]{\smash{{\SetFigFont{29}{34.8}{\rmdefault}{\mddefault}{\updefault}{\color[rgb]{0,0,0}\itr}%
}}}}
\end{picture}%

%% file: fig-pattern_translation.pdf_t
\begin{picture}(0,0)%
\includegraphics{fig-pattern_translation.pdf}%
\end{picture}%
\setlength{\unitlength}{4144sp}%
\begingroup\makeatletter\ifx\SetFigFont\undefined%
\gdef\SetFigFont#1#2#3#4#5{%
  \reset@font\fontsize{#1}{#2pt}%
  \fontfamily{#3}\fontseries{#4}\fontshape{#5}%
  \selectfont}%
\fi\endgroup%
\begin{picture}(3399,3399)(5389,-5698)
\put(5536,-2761){\makebox(0,0)[lb]{\smash{{\SetFigFont{29}{34.8}{\rmdefault}{\mddefault}{\updefault}{\color[rgb]{0,0,0}\itr}%
}}}}
\put(5536,-3436){\makebox(0,0)[lb]{\smash{{\SetFigFont{29}{34.8}{\rmdefault}{\mddefault}{\updefault}{\color[rgb]{0,0,0}\itr}%
}}}}
\put(5536,-4786){\makebox(0,0)[lb]{\smash{{\SetFigFont{29}{34.8}{\rmdefault}{\mddefault}{\updefault}{\color[rgb]{0,0,0}\itr}%
}}}}
\put(5536,-4111){\makebox(0,0)[lb]{\smash{{\SetFigFont{29}{34.8}{\rmdefault}{\mddefault}{\updefault}{\color[rgb]{0,0,0}\itr}%
}}}}
\put(5536,-5461){\makebox(0,0)[lb]{\smash{{\SetFigFont{29}{34.8}{\rmdefault}{\mddefault}{\updefault}{\color[rgb]{0,0,0}\itr}%
}}}}
\put(6256,-2761){\makebox(0,0)[lb]{\smash{{\SetFigFont{29}{34.8}{\rmdefault}{\mddefault}{\updefault}{\color[rgb]{0,0,0}\itr}%
}}}}
\put(6256,-3436){\makebox(0,0)[lb]{\smash{{\SetFigFont{29}{34.8}{\rmdefault}{\mddefault}{\updefault}{\color[rgb]{0,0,0}\itr}%
}}}}
\put(6256,-4786){\makebox(0,0)[lb]{\smash{{\SetFigFont{29}{34.8}{\rmdefault}{\mddefault}{\updefault}{\color[rgb]{0,0,0}\itr}%
}}}}
\put(6256,-4111){\makebox(0,0)[lb]{\smash{{\SetFigFont{29}{34.8}{\rmdefault}{\mddefault}{\updefault}{\color[rgb]{0,0,0}\itr}%
}}}}
\put(6256,-5461){\makebox(0,0)[lb]{\smash{{\SetFigFont{29}{34.8}{\rmdefault}{\mddefault}{\updefault}{\color[rgb]{0,0,0}\itr}%
}}}}
\put(6886,-2761){\makebox(0,0)[lb]{\smash{{\SetFigFont{29}{34.8}{\rmdefault}{\mddefault}{\updefault}{\color[rgb]{0,0,0}\itr}%
}}}}
\put(6886,-3436){\makebox(0,0)[lb]{\smash{{\SetFigFont{29}{34.8}{\rmdefault}{\mddefault}{\updefault}{\color[rgb]{0,0,0}\itr}%
}}}}
\put(6886,-4786){\makebox(0,0)[lb]{\smash{{\SetFigFont{29}{34.8}{\rmdefault}{\mddefault}{\updefault}{\color[rgb]{0,0,0}\itr}%
}}}}
\put(6886,-4111){\makebox(0,0)[lb]{\smash{{\SetFigFont{29}{34.8}{\rmdefault}{\mddefault}{\updefault}{\color[rgb]{0,0,0}\itr}%
}}}}
\put(6886,-5461){\makebox(0,0)[lb]{\smash{{\SetFigFont{29}{34.8}{\rmdefault}{\mddefault}{\updefault}{\color[rgb]{0,0,0}\itr}%
}}}}
\put(7561,-2761){\makebox(0,0)[lb]{\smash{{\SetFigFont{29}{34.8}{\rmdefault}{\mddefault}{\updefault}{\color[rgb]{0,0,0}\itr}%
}}}}
\put(7561,-3436){\makebox(0,0)[lb]{\smash{{\SetFigFont{29}{34.8}{\rmdefault}{\mddefault}{\updefault}{\color[rgb]{0,0,0}\itr}%
}}}}
\put(7561,-4786){\makebox(0,0)[lb]{\smash{{\SetFigFont{29}{34.8}{\rmdefault}{\mddefault}{\updefault}{\color[rgb]{0,0,0}\itr}%
}}}}
\put(7561,-4111){\makebox(0,0)[lb]{\smash{{\SetFigFont{29}{34.8}{\rmdefault}{\mddefault}{\updefault}{\color[rgb]{0,0,0}\itr}%
}}}}
\put(7561,-5461){\makebox(0,0)[lb]{\smash{{\SetFigFont{29}{34.8}{\rmdefault}{\mddefault}{\updefault}{\color[rgb]{0,0,0}\itr}%
}}}}
\put(8191,-2761){\makebox(0,0)[lb]{\smash{{\SetFigFont{29}{34.8}{\rmdefault}{\mddefault}{\updefault}{\color[rgb]{0,0,0}\itr}%
}}}}
\put(8191,-3436){\makebox(0,0)[lb]{\smash{{\SetFigFont{29}{34.8}{\rmdefault}{\mddefault}{\updefault}{\color[rgb]{0,0,0}\itr}%
}}}}
\put(8191,-4786){\makebox(0,0)[lb]{\smash{{\SetFigFont{29}{34.8}{\rmdefault}{\mddefault}{\updefault}{\color[rgb]{0,0,0}\itr}%
}}}}
\put(8191,-4111){\makebox(0,0)[lb]{\smash{{\SetFigFont{29}{34.8}{\rmdefault}{\mddefault}{\updefault}{\color[rgb]{0,0,0}\itr}%
}}}}
\put(8191,-5461){\makebox(0,0)[lb]{\smash{{\SetFigFont{29}{34.8}{\rmdefault}{\mddefault}{\updefault}{\color[rgb]{0,0,0}\itr}%
}}}}
\end{picture}%

%% file: dksigma.pdf_t
\begin{picture}(0,0)%
\includegraphics{dksigma.pdf}%
\end{picture}%
\setlength{\unitlength}{3947sp}%
\begingroup\makeatletter\ifx\SetFigFont\undefined%
\gdef\SetFigFont#1#2#3#4#5{%
  \reset@font\fontsize{#1}{#2pt}%
  \fontfamily{#3}\fontseries{#4}\fontshape{#5}%
  \selectfont}%
\fi\endgroup%
\begin{picture}(3164,2136)(1622,-1841)
\put(2652,-40){\makebox(0,0)[lb]{\smash{{\SetFigFont{10}{12.0}{\rmdefault}{\mddefault}{\updefault}{\color[rgb]{0,0,0}$\x_\cell$}%
}}}}
\put(3342,-1133){\makebox(0,0)[lb]{\smash{{\SetFigFont{10}{12.0}{\rmdefault}{\mddefault}{\updefault}{\color[rgb]{0,0,0}$\bfn_{\cell,\edge}$}%
}}}}
\put(2395,-609){\makebox(0,0)[lb]{\smash{{\SetFigFont{10}{12.0}{\rmdefault}{\mddefault}{\updefault}{\color[rgb]{0,0,0}$d_{\cell,\edge}$}%
}}}}
\put(3879,-675){\makebox(0,0)[lb]{\smash{{\SetFigFont{10}{12.0}{\rmdefault}{\mddefault}{\updefault}{\color[rgb]{0,0,0}$\edge$}%
}}}}
\put(3170,-513){\makebox(0,0)[lb]{\smash{{\SetFigFont{10}{12.0}{\rmdefault}{\mddefault}{\updefault}{\color[rgb]{0,0,0}$D_{K,\edge}$}%
}}}}
\put(3428,-1609){\makebox(0,0)[lb]{\smash{{\SetFigFont{10}{12.0}{\rmdefault}{\mddefault}{\updefault}{\color[rgb]{0,0,0}$\cell'$}%
}}}}
\put(2177,-1427){\makebox(0,0)[lb]{\smash{{\SetFigFont{10}{12.0}{\rmdefault}{\mddefault}{\updefault}{\color[rgb]{0,0,0}$\cell$}%
}}}}
\put(3936,-1345){\makebox(0,0)[lb]{\smash{{\SetFigFont{10}{12.0}{\rmdefault}{\mddefault}{\updefault}{\color[rgb]{0,0,0}$\x_{\cell'}$}%
}}}}
\put(3813,-986){\makebox(0,0)[lb]{\smash{{\SetFigFont{10}{12.0}{\rmdefault}{\mddefault}{\updefault}{\color[rgb]{0,0,0}$d_{\cell',\edge}$}%
}}}}
\end{picture}%

%% file: fig-grid_test1.pdf_t
\begin{picture}(0,0)%
\includegraphics{fig-grid_test1.pdf}%
\end{picture}%
\setlength{\unitlength}{4144sp}%
\begingroup\makeatletter\ifx\SetFigFont\undefined%
\gdef\SetFigFont#1#2#3#4#5{%
  \reset@font\fontsize{#1}{#2pt}%
  \fontfamily{#3}\fontseries{#4}\fontshape{#5}%
  \selectfont}%
\fi\endgroup%
\begin{picture}(4823,4824)(-11,-7173)
\end{picture}%

%% file: fig-grid_test2.pdf_t
\begin{picture}(0,0)%
\includegraphics{fig-grid_test2.pdf}%
\end{picture}%
\setlength{\unitlength}{4144sp}%
\begingroup\makeatletter\ifx\SetFigFont\undefined%
\gdef\SetFigFont#1#2#3#4#5{%
  \reset@font\fontsize{#1}{#2pt}%
  \fontfamily{#3}\fontseries{#4}\fontshape{#5}%
  \selectfont}%
\fi\endgroup%
\begin{picture}(4524,4524)(-11,-7173)
\end{picture}%

%% file: fig-initial_pattern_for_symmetry.pdf_t
\begin{picture}(0,0)%
\includegraphics{fig-initial_pattern_for_symmetry.pdf}%
\end{picture}%
\setlength{\unitlength}{4144sp}%
\begingroup\makeatletter\ifx\SetFigFont\undefined%
\gdef\SetFigFont#1#2#3#4#5{%
  \reset@font\fontsize{#1}{#2pt}%
  \fontfamily{#3}\fontseries{#4}\fontshape{#5}%
  \selectfont}%
\fi\endgroup%
\begin{picture}(760,515)(-972,-1917)
\put(-957,-1662){\makebox(0,0)[lb]{\smash{{\SetFigFont{12}{14.4}{\rmdefault}{\mddefault}{\updefault}{\color[rgb]{0,0,0}\itr:}%
}}}}
\end{picture}%

%% file: fig-initial_pattern_for_translation.pdf_t
\begin{picture}(0,0)%
\includegraphics{fig-initial_pattern_for_translation.pdf}%
\end{picture}%
\setlength{\unitlength}{4144sp}%
\begingroup\makeatletter\ifx\SetFigFont\undefined%
\gdef\SetFigFont#1#2#3#4#5{%
  \reset@font\fontsize{#1}{#2pt}%
  \fontfamily{#3}\fontseries{#4}\fontshape{#5}%
  \selectfont}%
\fi\endgroup%
\begin{picture}(656,472)(-779,-1962)
\put(-764,-1739){\makebox(0,0)[lb]{\smash{{\SetFigFont{10}{12.0}{\rmdefault}{\mddefault}{\updefault}{\color[rgb]{0,0,0}\itr:}%
}}}}
\end{picture}%

%% file: fig-est-w.pdf_t
\begin{picture}(0,0)%
\includegraphics{fig-est-w.pdf}%
\end{picture}%
\setlength{\unitlength}{3947sp}%
\begingroup\makeatletter\ifx\SetFigFont\undefined%
\gdef\SetFigFont#1#2#3#4#5{%
  \reset@font\fontsize{#1}{#2pt}%
  \fontfamily{#3}\fontseries{#4}\fontshape{#5}%
  \selectfont}%
\fi\endgroup%
\begin{picture}(4305,1908)(136,-1771)
\put(3451,-1486){\makebox(0,0)[lb]{\smash{{\SetFigFont{10}{12.0}{\rmdefault}{\mddefault}{\updefault}{\color[rgb]{0,0,0}$\overline{\x}_\cell$}%
}}}}
\put(1801,-811){\makebox(0,0)[lb]{\smash{{\SetFigFont{10}{12.0}{\rmdefault}{\mddefault}{\updefault}{\color[rgb]{0,0,0}$\x_\cell$}%
}}}}
\put(1726,-436){\makebox(0,0)[lb]{\smash{{\SetFigFont{10}{12.0}{\rmdefault}{\mddefault}{\updefault}{\color[rgb]{0,0,0}$R_+$}%
}}}}
\put(3976, 14){\makebox(0,0)[lb]{\smash{{\SetFigFont{10}{12.0}{\rmdefault}{\mddefault}{\updefault}{\color[rgb]{0,0,0}$\xi$}%
}}}}
\put(4426,-1261){\makebox(0,0)[lb]{\smash{{\SetFigFont{10}{12.0}{\rmdefault}{\mddefault}{\updefault}{\color[rgb]{0,0,0}$H_0$ ($x_d=0$)}%
}}}}
\put(1726,-1411){\makebox(0,0)[lb]{\smash{{\SetFigFont{10}{12.0}{\rmdefault}{\mddefault}{\updefault}{\color[rgb]{0,0,0}$R_-$}%
}}}}
\put(151,-1261){\makebox(0,0)[lb]{\smash{{\SetFigFont{10}{12.0}{\rmdefault}{\mddefault}{\updefault}{\color[rgb]{0,0,0}$\frac{\cter{cst:height}h_K}{3}$}%
}}}}
\put(1726,-1711){\makebox(0,0)[lb]{\smash{{\SetFigFont{10}{12.0}{\rmdefault}{\mddefault}{\updefault}{\color[rgb]{0,0,0}$Q_\cell$}%
}}}}
\put(2926,-886){\makebox(0,0)[lb]{\smash{{\SetFigFont{10}{12.0}{\rmdefault}{\mddefault}{\updefault}{\color[rgb]{0,0,0}${\cter{cst:height}h_K}$}%
}}}}
\put(151,-436){\makebox(0,0)[lb]{\smash{{\SetFigFont{10}{12.0}{\rmdefault}{\mddefault}{\updefault}{\color[rgb]{0,0,0}$\frac{\cter{cst:height}h_K}{3}$}%
}}}}
\end{picture}%

%% file: fig-cT.pdf_t
\begin{picture}(0,0)%
\includegraphics{fig-cT.pdf}%
\end{picture}%
\setlength{\unitlength}{3947sp}%
\begingroup\makeatletter\ifx\SetFigFont\undefined%
\gdef\SetFigFont#1#2#3#4#5{%
  \reset@font\fontsize{#1}{#2pt}%
  \fontfamily{#3}\fontseries{#4}\fontshape{#5}%
  \selectfont}%
\fi\endgroup%
\begin{picture}(2314,2257)(345,-2096)
\put(376,-1636){\makebox(0,0)[lb]{\smash{{\SetFigFont{10}{12.0}{\rmdefault}{\mddefault}{\updefault}{\color[rgb]{0,0,0}$\mathbi{a}_2$}%
}}}}
\put(901, 14){\makebox(0,0)[lb]{\smash{{\SetFigFont{10}{12.0}{\rmdefault}{\mddefault}{\updefault}{\color[rgb]{0,0,0}$\mathbi{a}_1$}%
}}}}
\put(1381,-522){\makebox(0,0)[lb]{\smash{{\SetFigFont{10}{12.0}{\rmdefault}{\mddefault}{\updefault}{\color[rgb]{0,0,0}$\edge_2$}%
}}}}
\put(764,-1175){\makebox(0,0)[lb]{\smash{{\SetFigFont{10}{12.0}{\rmdefault}{\mddefault}{\updefault}{\color[rgb]{0,0,0}$\edge_3$}%
}}}}
\put(1611,-1059){\makebox(0,0)[lb]{\smash{{\SetFigFont{10}{12.0}{\rmdefault}{\mddefault}{\updefault}{\color[rgb]{0,0,0}$\c_T$}%
}}}}
\put(2644,-1530){\makebox(0,0)[lb]{\smash{{\SetFigFont{10}{12.0}{\rmdefault}{\mddefault}{\updefault}{\color[rgb]{0,0,0}$\mathbi{a}_3$}%
}}}}
\put(1688,-2017){\makebox(0,0)[lb]{\smash{{\SetFigFont{10}{12.0}{\rmdefault}{\mddefault}{\updefault}{\color[rgb]{0,0,0}$\bfn_{T,\edge_1}$}%
}}}}
\put(2147,-383){\makebox(0,0)[lb]{\smash{{\SetFigFont{10}{12.0}{\rmdefault}{\mddefault}{\updefault}{\color[rgb]{0,0,0}$\bfn_{T,\edge_2}$}%
}}}}
\put(365,-536){\makebox(0,0)[lb]{\smash{{\SetFigFont{10}{12.0}{\rmdefault}{\mddefault}{\updefault}{\color[rgb]{0,0,0}$\bfn_{T,\edge_3}$}%
}}}}
\put(1999,-1384){\makebox(0,0)[lb]{\smash{{\SetFigFont{10}{12.0}{\rmdefault}{\mddefault}{\updefault}{\color[rgb]{0,0,0}$\edge_1$}%
}}}}
\end{picture}%